\documentclass[12pt]{article}

\usepackage{amssymb,amsmath}
\usepackage{amsthm}
\usepackage{hyperref}
\usepackage{mathrsfs}
\usepackage{textcomp}
\usepackage{enumerate}

\usepackage[normalem]{ulem}

\usepackage{todonotes}

\usepackage{tikz}
\usetikzlibrary{patterns}

\usepackage{verbatim}

\usepackage{sectsty}

\makeatletter

\newdimen\bibspace
\renewenvironment{thebibliography}[1]{%
 \section*{\refname %or \bibname if you use ``book'' as the documentclass
       \@mkboth{\MakeUppercase\refname}{\MakeUppercase\refname}}%
     \list{\@biblabel{\@arabic\c@enumiv}}%
          {\settowidth\labelwidth{\@biblabel{#1}}%
           \leftmargin\labelwidth
           \advance\leftmargin\labelsep
           \itemsep\bibspace
           \parsep\z@skip     %
           \@openbib@code
           \usecounter{enumiv}%
           \let\p@enumiv\@empty
           \renewcommand\theenumiv{\@arabic\c@enumiv}}%
     \sloppy\clubpenalty4000\widowpenalty4000%
     \sfcode`\.\@m}
    {\def\@noitemerr
      {\@latex@warning{Empty `thebibliography' environment}}%
     \endlist}

\makeatother

\makeatletter

\newtheorem{thm}{Theorem}[section]
\newtheorem{lem}[thm]{Lemma}
\newtheorem{prop}[thm]{Proposition}
\newtheorem{defn}[thm]{Definition}

\newtheorem{cor}[thm]{Corollary}
\newtheorem{rem}[thm]{Remark}

\numberwithin{equation}{section}

\def\XXint#1#2#3{{\setbox0=\hbox{$#1{#2#3}{\int}$}
  \vcenter{\hbox{$#2#3$}}\kern-.5\wd0}}

                \newcommand{\lda}{\lambda}
\newcommand{\om}{\Omega}                \newcommand{\pa}{\partial}
\newcommand{\va}{\varepsilon}           \newcommand{\ud}{\mathrm{d}}             
              \newcommand{\R}{\mathbb{R}}

\newcounter{marnote}

\begin{document}

\title{Regularity of viscosity solutions of the $\sigma_k$-Loewner--Nirenberg problem}

\author{\medskip  YanYan Li\footnote{ Partially supported by NSF grants DMS-1501004, DMS-2000261, and Simons Fellows Award 677077.}, \  \ Luc Nguyen,\ \ 
Jingang Xiong\footnote{Partially supported by NSFC 11922104 and 11631002.}}

\date{}

\maketitle

\begin{abstract}
We study the regularity of the viscosity solution $u$ of the $\sigma_k$-Loewner--Nirenberg problem on a bounded smooth domain $\Omega \subset \mathbb{R}^n$ for $k \geq 2$. It was known that $u$ is locally Lipschitz in $\Omega$. We prove that, with $d$ being the distance function to $\partial\Omega$ and $\delta > 0$ sufficiently small, $u$ is smooth in $\{0 < d(x) < \delta\}$ and the first $(n-1)$ derivatives of $d^{\frac{n-2}{2}} u$ are H\"older continuous in $\{0 \leq d(x) < \delta\}$. Moreover, we identify a boundary invariant which is a polynomial of the principal curvatures of $\partial\Omega$ and its covariant derivatives and vanishes if and only if  $d^{\frac{n-2}{2}} u$ is smooth in $\{0 \leq d(x) < \delta\}$. Using a relation between the Schouten tensor of the ambient manifold and the mean curvature of a submanifold and related tools from geometric measure theory, we further prove that, when $\partial\Omega$ contains more than one connected components, $u$ is not differentiable in $\Omega$. 
\end{abstract} 

\tableofcontents

\section{Introduction}

For a positive $C^2$ function $u$ defined on an open subset of $\mathbb{R}^n$, $n \geq 3$, let $A^u$ denote its conformal Hessian, namely 
\[
A^{u}= -\frac{2}{n-2}u^{-\frac{n+2}{n-2}} \nabla^2 u+\frac{2n}{(n-2)^2}u^{-\frac{2n}{n-2}} \nabla u \otimes \nabla u-\frac{2}{(n-2)^2}u^{-\frac{2n}{n-2}} |\nabla u|^2 I
\]
and let $\lambda(-A^u)$ denote the eigenvalues of  $-A^u$. 
For $1\le k\le n$, let $\sigma_k :\mathbb{R}^n \to \mathbb{R}$ be the $k$-th elementary symmetric function 
\[
\sigma_k (\lambda)= \sum_{1\leq i_1<\dots <i_k\leq n} \lambda_{i_1}\cdots \lambda_{i_k} 
\] 
and let $\Gamma_k$ denote the cone $\Gamma_k=\{\lambda=(\lambda_1,\dots ,\lambda_n):\sigma_1( \lambda )>0, \dots, \sigma_k(\lambda) >0\}$. Let $\Omega$ be a bounded domain in $\mathbb{R}^n$ and consider the $\sigma_k$-Loewner--Nirenberg problem in $\Omega$, i.e. the problem of finding a positive solution to
\begin{equation}\label{eq:main-1a}
\sigma_k (\lambda(-A^u))= N_k, \quad \lambda(-A^u) \in \Gamma_k, \quad u>0 \quad \mbox{in }\Omega, 
\end{equation}
\begin{equation} \label{eq:main-1b}
u(x)\to \infty \quad \mbox{as }d(x) \to 0, 
\end{equation}
where $N_k=2^{-k} \tbinom{n}{k}$ and $d(x)=\mathrm{dist}(x,\partial \Omega)$. The geometric nature of problem \eqref{eq:main-1a}--\eqref{eq:main-1b} comes from the fact that $u^{\frac{4}{n-2}} A^u_{ij} \ud x^i \ud x^j$ is the Schouten tensor of the conformally flat metric $u^{\frac{4}{n-2}}\mathring{g}$ on $\Omega \subset \mathbb{R}^n$, where $\mathring{g}$ is the Euclidean metric.

When $k = 1$, problem \eqref{eq:main-1a}--\eqref{eq:main-1b} reduces to
\begin{equation}\label{eq:sm-1a}
\Delta u = \frac{n(n-2)}{4} u^{\frac{n+2}{n-2}}, \quad u>0 \quad \mbox{in }\Omega, \quad  u(x)\to \infty \quad \mbox{as }d(x) \to 0.
\end{equation}
This problem was first studied in the classical paper of Loewner and Nirenberg \cite{LoewnerNirenberg} where, among other results, the existence of a unique smooth positive solution was proved when the  boundary of $\Omega$ is smooth and compact. Since then, further studies of problem \eqref{eq:sm-1a} and its generalization in manifold settings have been done by many authors; see e.g. Allen, Isenberg, Lee and Allen \cite{AILA18}, Andersson, Chru\'sciel and Friedrich \cite{Andersson-Chrusciel-Friedrich}, Aviles \cite{Aviles82-CPDE}, Aviles and McOwen \cite{AvilesMcOwen88}, Finn \cite{Finn98-CPDE}, Gover and Waldron \cite{GoverWaldron17}, Graham \cite{Graham17-ProcAMS}, Han, Jiang and Shen \cite{HanJiangShen20}, Han and Shen \cite{HanShen20-JFA}, Jiang \cite{Jiang21-JFA}, Mazzeo \cite{Mazzeo91-Indiana}, V\'eron \cite{Veron81-JDE} and the references therein.

When $2 \leq k \leq n$, the $\sigma_k$-Loewner--Nirenberg problem \eqref{eq:main-1a}--\eqref{eq:main-1b} is a fully nonlinear (non-uniformly) elliptic problem of Hessian type. Some key results for this problem and its analogues on manifolds have been obtained by Chang, Han and Yang \cite{ChangHanYang05-JDE}, Gonz\'{a}lez, Li and Nguyen \cite{GonLiNg}, Guan \cite{Guan:negative-Ricci}, Gursky, Streets and Warren \cite{Gursky-Streets-Warren}, Gurksy and Viaclovsky \cite{Gursky-Viaclovsky:negative-curvature}, and Li and Nguyen \cite{LiNg21-JMS}. For other related works, see also Li and Sheng \cite{Li-Sheng:flow},  Sui \cite{Sui},  Wang \cite{Wang21-CVPDE} and the references therein.

When $\Omega$ is an annulus $\{a<|x|<b\}$ for some positive constants $a<b$, it was shown in Chang, Han and Yang \cite{ChangHanYang05-JDE} that \eqref{eq:main-1a}--\eqref{eq:main-1b} has no rotationally symmetric $C^2$ solution. 

In a closely related context, Li in \cite{Li09-CPAM} gave the definition of viscosity solutions to nonlinear Yamabe-type equations, proved comparision principles for Lipschitz viscosity solutions and in turn established Liouville-type theorems and local gradient estimates for solutions of  general nonlinear Yamabe-type equations.  Such comparison principles  were strengthened by Li, Nguyen and Wang in \cite{LiNgWang18-CVPDE} showing that they hold for continuous viscosity solutions and also  for a larger class of equations. Based on the latter comparison principles, Gonz\'{a}lez, Li and Nguyen proved in \cite{GonLiNg} the uniqueness of continuous viscosity solutions to the $\sigma_k$-Loewner--Nirenberg problem for general $\Omega$ with smooth boundary.  The combination of this uniqueness result and the above mentioned result of Chang, Han and Yang implies that  there is no $C^2$ solution on any annulus. 

Gonz\'{a}lez, Li and Nguyen also proved in \cite{GonLiNg} the existence of a locally Lipschitz viscosity solution to the $\sigma_k$-Loewner--Nirenberg problem for general $\Omega$ with smooth boundary. Furthermore, by the same work, the solution $u$ satisfies 
\begin{equation} \label{eq:main-1c}
d(x) u(x)^{\frac{2}{n-2}}\to 1 \quad \mbox{as }d(x) \to 0. 
\end{equation} 

It was then shown in Li and Nguyen \cite{LiNg21-JMS} that in the case of an annulus $\om = \{a<|x|<b\}$, \eqref{eq:main-1a}--\eqref{eq:main-1b} admits no $C^1$ solution: Speficially, the unique viscosity solution of \eqref{eq:main-1a}--\eqref{eq:main-1b} is not $C^1$ across the sphere $S_* = \{|x| = \sqrt{ab}\}$, is smooth in $\om \setminus S_*$, and is $C^{1, \frac{1}{k}}$ but not $C^{1,\gamma}$ for any $\gamma > \frac{1}{k}$ in each of $\{a < |x| \leq \sqrt{ab}\}$ and $\{\sqrt{ab} \leq |x| < b\}$.

More generally, it is worthwhile to understand for what domain $\Omega$ problem \eqref{eq:main-1a}--\eqref{eq:main-1b} admits a $C^2$ solution and for what domain $\Omega$ it admits no $C^2$ solution. For this purpose, several specific questions were raised in \cite{LiNg21-JMS}. Our results in the sequel are efforts in addressing this issue, and in particular affirmatively confirm \cite[Conjecture 1.1]{LiNg21-JMS} and give a complete answer to \cite[Question 1.3]{LiNg21-JMS}.

Our first theorem concerns the non-existence of $C^1$ solutions for $k \geq 2$ when $\partial\Omega$ is disconnected. We would like to point out that the situation is very different when $k = 1$: The solution of the Loewner--Nirenberg problem is smooth in $\Omega$.

\begin{thm} \label{thm:main} 
Let $\Omega$ be a bounded domain in $\R^n$, $n \geq 3$, with smooth and disconnected boundary $\partial \Omega$. Suppose that $2 \leq k \leq n$.
Then problem \eqref{eq:main-1a}--\eqref{eq:main-1b} has no positive solution in $C^1(\Omega)$. In particular, the unique locally Lipschitz viscosity solution of \eqref{eq:main-1a}--\eqref{eq:main-1b} does not belong to $C^1(\Omega)$. 
\end{thm}

Some ideas supporting Theorem \ref{thm:main} were described in \cite{LiNg21-JMS}. Roughly speaking, they are as follows. On one hand, by \cite[Theorem 1.1]{LiNg21-JMS}, the complete Riemannian manifold $(\Omega, u^{\frac{4}{n-2}}\mathring{g})$, if smooth, cannot admit a minimal immersion.  On the other hand, in view of the fact that $u$ behaves like $d^{-\frac{n-2}{2}}$ near $\partial\Omega$, the hypersurface $\Sigma_\delta = \{x \in \Omega: d(x) = \delta\}$ is expected to be mean convex with respect to $u^{\frac{4}{n-2}}\mathring{g}$ for small $\delta > 0$ and so there should be a minimal hypersurface in $\Omega$ confined by $\Sigma_\delta$. In fact, this line of argument gives the following result for the differential inclusion $\lambda(-A^u) \in \bar\Gamma_2$:

\begin{thm}\label{thm:MX}
Let $\Omega$ be a bounded domain in $\R^n$, $n \geq 3$, with smooth and disconnected boundary $\partial \Omega$. There exists no positive function $u \in C^1(\bar\Omega)$ such that $\lambda(-A^u) \in \bar\Gamma_2$ in $\Omega$ in the viscosity sense and $\partial\Omega$ has non-negative mean curvature $H_{\partial\Omega} \geq 0$ with respect to $g = g_u$ and the normal pointing towards $\Omega$.
\end{thm}

It follows from the proof that in Theorem \ref{thm:MX} it is enough to assume that $\partial\Omega$ is  $C^{1,\alpha} \cap W^{2,p}$ regular for some $\alpha \in (0,1]$ and $p \geq 2$.

In the proof of Theorem \ref{thm:MX}, the assumption that $u$ is differentiable is used in the application of Corollary \ref{Cor:Obs1} and in asserting that the hypersurface $\partial E \setminus \mathcal{S}$ in Proposition \ref{prop:Conc} indeed has zero mean curvature. The rest of the proof uses only that $u$ is Lipschitz continuous. Nevertheless, the assumption $u \in C^1(\bar\Omega)$ cannot be relaxed to $u \in C^{0,1}(\Omega)$, even when one imposes the stronger assumption that $\lambda(-A^u) \in \Gamma_2$ in $\Omega$ and $H_{\partial\Omega} > 0$. Indeed, it can be easily deduced from \cite[Theorem 1.2]{LiNg21-JMS} that, for any $0 < a < c < b < \infty$ and with $\Omega = \{a < |x| < b\}$,  there exists $u \in C^{0,1}(\Omega) \cap C^\infty(\{a \leq |x| < c\}) \cap C^\infty(\{c < |x| \leq b\})$ such that $\lambda(-A^u) \in \Gamma_2$ in $\{a < |x| < b\}$ in the viscosity sense and that $H_{\partial\Omega} > 0$ with respect to $g = g_u$ and the normal pointing towards $\Omega$. In the particular case of this example, if one follows the strategy of our proof, the minimal set $E$ one would obtain in the context of Proposition \ref{prop:Conc} is the set $\{|x| < c\}$ which has smooth boundary $\partial E = \{|x| = c\}$. However, $u$ fails to be differentiable exactly on this set, the mean curvature of $\partial E$ with respect to $g_u$ is undefined, and Corollary \ref{Cor:Obs1} is inapplicable there.

Our next theorem shows that, for small $\delta > 0$, the solution $u$ of \eqref{eq:main-1a}--\eqref{eq:main-1b} is smooth in $\{0 < d(x) < \delta\}$, and confirms the aforementioned expectation that $\Sigma_\delta = \{x \in \Omega: d(x) = \delta\}$ is mean convex with respect to $u^{\frac{4}{n-2}}\mathring{g}$, which is needed in the passage from Theorem \ref{thm:MX} to Theorem \ref{thm:main}. Let $\pi$ be a smooth map defined in a neighborhood of $\partial\Omega$ such that $\pi(x)$ is the point on $\partial\Omega$ which is closest to $x$, i.e. $|x-\pi(x)|=d(x)$ near $\partial\Omega$. We show that, near $\partial\Omega$, $\ln(d\,u^{\frac{2}{n-2}})$ admits an expansion of the form
\begin{multline}
\ln\big(d(x)u(x)^{\frac{2}{n-2}}\big)  \sim   \sum_{p = 1}^{n-1} (c_{p,0} \circ \pi)(x) \,d(x)^p + (c_{n,1}\circ \pi)(x) \,d(x)^n\,\ln d(x) + (c_{n,0} \circ \pi)(x) \,d(x)^n\\
	 +  \sum_{p = n+1}^\infty \sum_{q=0}^{N_p} (c_{p,q} \circ \pi)(x) \,d(x)^p\,(\ln d(x))^q.
	\label{Eq:uPHExp}
\end{multline}
Here \eqref{Eq:uPHExp} is understood in the sense that, for every $m \geq 1$ and with $W_m = \sum_{p = 1}^m \sum_{q=0}^{N_p} (c_{p,q} \circ \pi) \,d^p\,(\ln d)^q$,
\begin{equation} \label{Eq:YYwWmEst}
\Big|\nabla_T^s \nabla^j \big(\ln\big(d(x)u(x)^{\frac{2}{n-2}}\big)  - W_m(x)\big)\Big| = O(d(x)^{m-j+\gamma}) \quad \text{ as } d(x) \rightarrow 0
\end{equation}
for $j \geq 0$, $s \geq 0$, $\gamma \in (0,1)$, where the implicit constant in \eqref{Eq:YYwWmEst} depends only on $\Omega$, $m$, $j$, $s$, $\gamma$. Here $\nabla_T$ denotes the gradient along the hypersurfaces orthogonal to $\nabla d$. It is clear that such expansion is unique.

\begin{thm} \label{thm:mainX} 
Let $\Omega$ be a bounded domain in $\R^n$, $n \geq 3$, with smooth boundary $\partial \Omega$. Then there exists $\delta_1 = \delta_1(n,k,\Omega) > 0$ such that the viscosity solution $u$ of \eqref{eq:main-1a}--\eqref{eq:main-1b} satisfies 
\[
d\,u^{\frac{2}{n-2}} \in C^\infty(\{0 < d(x) < \delta_1\}) \cap C^{n-1,\gamma}(\{0 \leq d(x) \leq \delta_1\}) \text{ for every } \gamma \in (0,1).
\]
Moreover, near $\partial\Omega$, $d\,u^{\frac{2}{n-2}}$ admits an expansion of the form \eqref{Eq:uPHExp} in the sense that \eqref{Eq:YYwWmEst} holds for every $m \geq 1$ where coefficients $c_{p,q}$ are smooth functions on $\partial\Omega$. 
Finally, if $c_{n,1} = 0$ in some $\partial\Omega \cap B(x_0,r_0)$, then $d\,u^{\frac{2}{n-2}}$ is smooth in $\{0 \leq d(x) \leq \delta_1\} \cap B(x_0,r_0) $.
\end{thm}

\begin{rem}
Let $\kappa_1, \ldots, \kappa_{n-1}$ denote the principal curvatures of $\partial\Omega$. It is seen from the proof of the theorem that $c_{n,1}$ can be written as an explicit polynomial in terms of the principal curvatures $\kappa_j$'s and their covariant derivatives up to order $2(n-1)$. \end{rem}

\begin{rem}
By the last assertion of the theorem, $c_{n,1} \equiv 0$ on $\partial\Omega$ if and only if $du^{\frac{2}{n-2}}$ is smooth in $\{0 \leq d(x) \leq \delta_1\}$. In addition, by the conformal invariance of \eqref{eq:main-1a}--\eqref{eq:main-1b}, the local vanishing property of $c_{n,1}$ is invariant under M\"obius transformations.
\end{rem}

When $k=1$, Theorem \ref{thm:mainX} was proved by Andersson, Chru\'sciel and Friedrich \cite{Andersson-Chrusciel-Friedrich} and Mazzeo \cite{Mazzeo91-Indiana}. See Han, Jiang and Shen \cite{HanJiangShen20} and Han and Shen \cite{HanShen20-JFA} for results when $\Omega$ is not a smooth domain. Also in the case $k = 1$, it was shown by Gover and Waldron \cite{GoverWaldron17} and by Graham \cite{Graham17-ProcAMS} via a volume renormalization procedure that the obstruction for the smoothness of $d u^{\frac{2}{n-2}}$ (i.e. the vanishing of the coefficient $c_{n,1}$) can be characterized using a conformally invariant energy, which can be viewed as a higher dimensional analogue of the Willmore energy. There has been an interest in understanding if analogous results can be obtained in fully nonlinear settings. For related discussions, see e.g. Fefferman \cite{Fefferman-AnnM76}, Fefferman and Graham \cite{FeffermanGraham-A85}, Gursky, Han and Stolz \cite{GurskyHanStolz21-TrAMS} and the references therein.

For a discussion about the proof of Theorem \ref{thm:mainX} and related works in the literature, see Subsection \ref{SSec:2-1}.

We conclude the introduction with a more general version of Theorem \ref{thm:main} and Theorem \ref{thm:mainX}. Let
\begin{align}  
&\Gamma\subset \Bbb R^n \mbox{ be an open convex symmetric
cone with vertex at  
the origin,} 
	\label{fG1} \\
&\{\lambda \in \Bbb R^n | \lambda_i > 0, 1 \leq i \leq n\} \subset
\Gamma \subset \{\lambda \in \Bbb R^n | \lambda_1 + \ldots + \lambda_n >
0\},
	\label{fG2}\\
&f\in C^\infty (\Gamma) \cap C^0 (\overline{\Gamma}) \mbox{ be concave,
homogeneous of degree one, and
symmetric},
	\label{fG3}\\
&f>0\ \mbox{in}\ \Gamma,
\quad f = 0 \mbox{ on } \partial\Gamma ; \quad
f_{\lambda_i} > 0 \ \mbox{in
} \Gamma  \quad \forall 1 \leq i \leq n.
\label{fG4}
\end{align}
In \eqref{fG3}, we say that $f$ is symmetric if $f(\lambda_1, \ldots, \lambda_n) = f(\lambda_{\sigma(1)}, \ldots, \lambda_{\sigma(n)})$ for any permutation $\sigma$. It follows from a more general result of \cite{GonLiNg} that the problem
\begin{equation}\label{eq:main-fGa}
f (\lambda(-A^u))= 1, \quad \lambda(-A^u) \in \Gamma, \quad u>0 \quad \mbox{in }\Omega, 
\end{equation}
\begin{equation} \label{eq:main-fGb}
u(x)\to \infty \quad \mbox{as }d(x) \to 0.
\end{equation}
has a unique continuous viscosity solution $u$. Moreover, $u \in C^{0,1}_{\rm loc}(\Omega)$. In addition, under the normalization $f(\frac{1}{2}, \ldots, \frac{1}{2}) = 1$, \eqref{eq:main-1c} holds. The proof of Theorem \ref{thm:main} and Theorem \ref{thm:mainX} can be adapted to yield the following result.

\begin{thm} \label{thm:mainfG} 
Let $\Omega$ be a bounded domain in $\R^n$, $n \geq 3$, with smooth boundary $\partial \Omega$, and let $(f,\Gamma)$ satisfy \eqref{fG1}--\eqref{fG4} together with the normalization $f(\frac{1}{2}, \ldots, \frac{1}{2}) = 1$. Then the conclusion of Theorem \ref{thm:mainX} holds for the viscosity solution $u$ of \eqref{eq:main-fGa}--\eqref{eq:main-fGb}. In addition, if $\Gamma \subset \Gamma_2$, then the conclusion of Theorem \ref{thm:main} holds.
\end{thm}

The rest of the paper is organized as follows. We start with the proof of Theorem \ref{thm:mainX} and the first statement of Theorem \ref{thm:mainfG} in Section \ref{Sec:BdrExp}. We then give the proof of Theorem \ref{thm:main}, Theorem \ref{thm:MX} and the second statement of Theorem \ref{thm:mainfG} in Section \ref{Sec:MainProof}.

\bigskip

\noindent {\bf Acknowledgement:} The authors would like to thank Fang-Hua Lin for stimulating discussions. Part of this work was completed while J. Xiong was visiting Rutgers University, to which he is grateful for providing very stimulating research environments and supports. 

\section{Regularity near the boundary: Theorem \ref{thm:mainX} and its generalization}\label{Sec:BdrExp}

In this section, we prove Theorem \ref{thm:mainX} and the first part of Theorem \ref{thm:mainfG}.

\subsection{The setup and main ideas of the proof}\label{SSec:2-1}

Before discussing the strategy of the proof of Theorem \ref{thm:mainX} (and its generalization in Theorem \ref{thm:mainfG}), some comments are in order. In fully nonlinear settings, closest to our result above are those of Gursky, Streets and Warren \cite{Gursky-Streets-Warren} and Wang \cite{Wang21-CVPDE} for the equation $\sigma_k (\lambda(-Ric_{\tilde g})) = 1$ on a given compact manifold $M$ with non-empty boundary $\partial M$  for an unknown metric $\tilde g$ which is conformal to a given Riemannian metric $g$ on $M$ and whose conformal factor blows up at $\partial M$. It is known that this equation has better ellipticity property than the $\sigma_k$-Loewner--Nirenberg equation \eqref{eq:main-1a}. In \cite{Gursky-Streets-Warren}, it was established that this problem has a unique solution which is smooth in $M \setminus \partial M$. In \cite{Wang21-CVPDE}, a finite boundary expansion as in \eqref{Eq:uPHExp} up to term of order $O(d^{n}\,\ln d)$ was established in $C^0$ sense. In these works as well as works cited earlier on in the case $k = 1$, the smoothness of the solution in $\Omega$ or $M \setminus \partial M$ plays a role. In contrast, in our case, smoothness of solution is not available a priori; in fact global differentiability fails in light of Theorem \ref{thm:main}. Thus, a novelty of our proof of Theorem \ref{thm:mainX} concerns not only the smoothness of the solution near the boundary but also the quantitative nature of estimate \eqref{Eq:YYwWmEst}. In fact, even if it is given a priori that the solution is smooth near the boundary or in the whole domain, we do not know of a simpler proof. Of importance in our argument is our use of a result of Savin \cite{Savin07-CPDE} for small perturbation solutions for elliptic equations to obtain certain needed quantitative $C^2$ estimate near the boundary. See below for more details.

 Let $u=e^{\frac{n-2}{2} w}$ and 
\[
S(w)= -\nabla^2 w + \nabla w \otimes \nabla w -\frac{1}{2}|\nabla w|^2 I. 
\]
Then we see that 
\begin{align*}
A^{u}
	& =-\frac{2}{n-2} e^{-2 w}\Big(\frac{n-2}{2 }\nabla^2 w+ \frac{(n-2)^2}{4} \nabla w\otimes \nabla w\Big)\\
	&\qquad	+\frac{n}{2} e^{-2w}  \nabla w\otimes \nabla w -\frac{1}{2} e^{-2w} |\nabla w|^2\\
	&=e^{-2w} S(w). 
\end{align*}
Thus problem  \eqref{eq:main-1a}--\eqref{eq:main-1b} can be rewritten as 
\begin{equation} \label{eq:main-2a}
\sigma_k (\lambda (-S(w)))= N_k  e^{2kw}, \quad   \lambda(-S(w)) \in \Gamma_k \quad \mbox{in }\Omega,
\end{equation}
\begin{equation} \label{eq:main-2b}
w(x)\to \infty \quad \mbox{as }d(x) \to 0,
\end{equation} and \eqref{eq:main-1c} is equivalent to  
\begin{equation} \label{eq:main-2c}
\lim_{d(x)\to 0} |w(x)+\ln d(x) |=0. 
\end{equation}

For any function $v \in C^2(\Omega)$, define 
\begin{align*}
G(v)
	&:=(\sigma_k (\lambda (-S(v)))- N_k  e^{2kv})  d^{2k}= \sigma_k (\lambda (-d^2S(v)))- N_k d^{2k} e^{2kv}.
\end{align*}

Let $\Omega_{\delta_0} = \{x\in \Omega: d(x) > \delta_0\}$ with $\delta_0>0$ sufficiently small such that $d(x) \in C^\infty(\overline{\Omega \setminus \Omega_{\delta_0}})$, and $\pi: \overline{\Omega \setminus \Omega_{\delta_0}} \rightarrow \partial\Omega$ be the orthogonal projection map, i.e. $|x-\pi(x)|=d(x)$.

The idea of the proof of Theorem \ref{thm:mainX} is roughly as follows.

\medskip
\noindent\underline{Step 1:} We first show that there is a function of the form
\begin{equation}
W = -\ln d+\sum_{p=1}^{n-1} (c_{p,0} \circ \pi) d^p + (c_{n,1} \circ \pi) d^n\ln d
	\label{Eq:01II22-W1}
\end{equation}
such that $\lambda(-S(W)) \in \Gamma_k$ and $G(W) = o(d^{n})$ near $\partial\Omega$. Moreover, for any $\mu \in C^\infty(\partial\Omega)$, there exists a unique sequence $W_n =  W + (\mu \circ \pi) d^n$, $W_{n+1}$, \ldots of smooth functions near $\partial\Omega$, depending on $\mu$, with the following properties:
\begin{align*}
W_m - W_{m-1} &= o(d^{m-1})  \text{ near } \partial\Omega \text{ for } m \geq n+1,\\
\lambda(-S(W_m)) &\in \Gamma_k \text{ and } G(W_m) = o(d^m) \text{ near } \partial\Omega  \text{ for } m \geq n. 
\end{align*}
See Lemma \ref{lem:Wm-1Wm} and Proposition \ref{prop:expansion-1}. Similar constructions have appeared in other contexts, see e.g. \cite{Andersson-Chrusciel-Friedrich, Gursky-Streets-Warren, Wang21-CVPDE}. 

\medskip
\noindent\underline{Step 2:} Using a barrier argument, we obtain $w - W = O(d^n)$ at $C^0$ level. This gives the first $n$ terms in the expansion  \eqref{Eq:uPHExp} at $C^0$ level. We then use a result of Savin \cite{Savin07-CPDE} for small perturbation solutions for elliptic equations to show that $w - W = O(d^n)$ also at derivative level -- see Proposition \ref{prop:PreOpExp}. This gives local smoothness near the boundary of $u$ and the $C^{n-1,\gamma}$-regularity of $d^{\frac{n-2}{2}}u$ up to the boundary.

To explain briefly how Savin's result is applied, let us consider a point $x_0 \in \partial\Omega$ near which we would like to prove our estimate. Without loss of generality, we may assume that $x_0$ is the origin, the $x_n$ axis point toward $\Omega$ and the axes $x_1, \ldots, x_{n-1}$ are along directions tangential to $\partial\Omega$. Then the function $h(x) := -\ln x_n$ satisfies $G(h) = 0$ in $\{x_n > 0\}$. This leads us to consider the function $\hat w(x) = -h(x) + w(x)$ which satisfies 
\[
\hat G[w]
	:= \sigma_k (\lambda(-\hat S(\hat w)) - N_k e^{2k \hat w} 
		= 0,
\]
\[
\hat S(\hat w)(x)
	:= x_n^2 S(w)(x) 
	= x_n^2 S(\hat w)(x)
		- x_n (e_n \otimes \nabla \hat w +\nabla \hat w \otimes  e_n)
		+ x_n e_n \partial_n \hat w I
		- \frac{1}{2} I.
\]

Clearly $\hat G$ is smooth and, as $G(h) = 0$, $0$ is a solution for $\hat G$, i.e. $\hat G[0] = 0$, and $\hat G$ is elliptic near $0$. Note that $\hat G$ has the following scaling property: If we define $\hat w_r(x) = \hat w(x/r)$ for $r > 0$, then
\[
\hat G[\hat w_r](x)
	= \hat G[\hat w](x/r) = 0.
\]
Now, the fact that $|w - W| = O(d^n)$ obtained in Step 1 implies that, for large $r$, $|\hat w_r| \leq \frac{C}{r}$ in a small ball centered at $(0, \ldots, 0, 1)$. We may then apply Savin's result to $\hat w_r$, yielding the $C^{2,\alpha}$ regularity of $\hat w_r$ in a smaller ball around $(0, \ldots, 0, 1)$. Returning to $w$, this shows the $C^{2,\alpha}$ regularity of $w$ near $\partial\Omega$.

\medskip
\noindent\underline{Step 3:} The estimates obtained in Step 2 have the consequence that, along the integral curves of $\nabla d$, the equation $G(w) = 0$ can be recast as an ODE in the form
\[
d^2 w''(d)  - (n-2) d w'(d) - n w(d) = \text{correction terms},
\]
from which one can deduce the existence of the coefficient function $c_{n,0}$ and hence all higher coefficient functions in the expansion \eqref{Eq:uPHExp}.

\subsection{Step 1: The coefficient functions $c_{1,0}, \ldots, c_{n-1,0}$ and $c_{n,1}$}

The starting point is the following lemma.

\begin{lem}\label{lem:W0}
Let $\Omega$ be a bounded domain in $\R^n$, $n \geq 3$, with smooth boundary $\partial \Omega$. Let $\kappa_1, \ldots, \kappa_{n-1}$ denote the principal curvatures of $\partial\Omega$. Then the function $W_0 := -\ln d$ satisfies $\lambda(-S[W_0]) \in \Gamma_k$ near $\partial\Omega$ and
\[
G(W_0)
= \sigma_k\Big(\frac{1}{2}+ \frac{\kappa_1 \circ \pi d}{1-\kappa_1 \circ \pi  d}, \ldots, \frac{1}{2}+ \frac{\kappa_i \circ \pi  d}{1-\kappa_i \circ \pi  d}, \frac{1}{2}\Big) - N_k.
\]
In particular, in a neighborhood of $\partial\Omega$, $G(W_0) = O(d)$, and $G(W_0)$ is smooth and can be written as a convergent power series of $d$:
\[
G(W_0)(x) = \sum_{p=1}^\infty (a^{(0)}_{p} \circ \pi)(x) d(x)^p
\]
where $a^{(0)}_{p} $ is an explicitly computable polynomial of the principal curvatures $\kappa_1, \ldots, \kappa_{n-1}$.
\end{lem}

\begin{proof}
Pick an arbitrary point $\bar x\in \om \setminus \om_{\delta_0} $. Without loss of generality we assume $\bar x=(\bar x',\bar x_n)=(0',\bar x_n)$, $\bar x_n> 0$ and $\pi(\bar x)=0\in \partial \Omega$. Following \cite[Section 14.6]{GT}, we assume that the coordinates system $\{x',x_n\}$ is a principal coordinates system at $0$, and 
\[
\om \cap B_{r_0}=  \{x\in B_{r_0}: x_n>\varphi(x')\},
\]
where 
\[
\varphi(0')=|\nabla _{x'} \varphi(0')|=0, \quad \nabla^2_{x'} \varphi(0')= \mathrm{diag}[\kappa_1, \dots, \kappa_{n-1}].
\] 
Then we have, at $\bar x$,
\begin{equation} \label{eq:distance}
\begin{split}
\nabla d&=(0',1)=:e_n,\\ 
\nabla^2 d&=\mathrm{diag}\left[\frac{-\kappa_1}{1-\kappa_1 d}, \dots, \frac{-\kappa_{n-1}}{1-\kappa_{n-1} d},0 \right ].
\end{split}
\end{equation}
(See \cite[Lemma 14.17]{GT}.) In particular, 
\[
\pa_{ii}^2d = -\kappa_i \sum_{p=0}^\infty(\kappa_i d)^p \quad \mbox{for }1\le i\le n-1.
\]
Similarly, at $\bar x$,
\[
\pa_i \pi=e_i \quad \mbox{if }i<n, \quad \pa_n \pi=0,
\]
\[
\pa_{ij} ^2\pi= -e_n \pa_{ij} ^2 d \quad \mbox{if }i+j<2n, \quad \pa_{nn}\pi =0. 
\]
It follows that, at $\bar x$, 
\begin{align*}
\pa_i W_0
	&= \begin{cases}
-\frac{1}{d}  & \mbox{if }i=n,\\[2mm]
0  & \mbox{otherwise}, 
\end{cases}\\
\pa_{ij}^2 W_0
	&= \begin{cases} 
-\frac{\pa_{ii}^2 d}{d} &\mbox{if }i= j<n, \\[2mm] 
 \frac{1}{d^2}  & \mbox{if }i=j=n,\\[2mm]
0 & \mbox{otherwise}.
\end{cases}
\end{align*} 
Hence, 
\begin{align}
-d^2 S_{ij}(W_0)
	&= d^2 \pa_{ij}^2 W_0 -d^2 \pa_i W_0 \pa_j W_0 +\frac{d^2}{2} |\nabla W_0 |^2 \delta_{ij} \nonumber  \\
	&=\begin{cases}
\frac{1}{2}+ \frac{\kappa_i d}{1-\kappa_i d} & \mbox{if }i=j<n, \\[2mm]
\frac{1}{2} & \mbox{if }i=j=n,\\[2mm]
0 & \mbox{otherwise}. 
\end{cases}
\label{Eq:SW0}
\end{align}
Clearly, this implies that $\lambda(-S(W_0)) \in \Gamma_k$ near $\partial\Omega$ and 
\begin{align*}
\sigma_k(\lambda(-d^2 S(W_0))) 
= \sigma_k\Big(\frac{1}{2}+ \frac{\kappa_1 d}{1-\kappa_1 d}, \ldots, \frac{1}{2}+ \frac{\kappa_i d}{1-\kappa_i d}, \frac{1}{2}\Big).
\end{align*}
The conclusion follows.
\end{proof}

It will be convenient to use the following notations.

\begin{defn}\label{def:polyhom}
Let $\Omega$ be a bounded domain in $\mathbb{R}^n$, $n \geq 3$, with smooth boundary $\partial\Omega$. We denote by $\mathcal{W}$ the set of functions $f$ defined in a neighborbood of $\partial\Omega$ for which there exist a sequence $\{N_p\}$ of non-negative integers and a sequence $\{a_{p,q}\}_{p \geq 0, 0 \leq q \leq N_p}$ of smooth functions on $\partial\Omega$ such that, for every $m \geq 0$, $\gamma \in (0,1)$, $s, j \geq 0$,
\[
\Big|\nabla_T^s \nabla^j \Big(f - \sum_{p=0}^m \sum_{q=0}^{N_p} (a_{p,q} \circ \pi)\, d^{p} \,(\ln d)^q\Big)\Big| = O(d^{m- j+\gamma}) \text{ as } d(x) \rightarrow 0,
\]
where $\nabla_T$ denotes the gradient along the hypersurfaces orthogonal to $\nabla d$. Functions in $\mathcal{W}$ are sometimes known as polyhomogeneous functions. For $\ell \geq 0$, we denote by $\mathcal{W}^\ell$ the set of functions $f \in \mathcal{W}$ satisfying $f = o(d^\ell)$ near $\partial\Omega$, i.e. the coefficients $a_{p,q}$ as above vanish for all $0 \leq p \leq \ell$.
\end{defn}

By Lemma \ref{lem:W0}, $G(W_0) \in \mathcal{W}^0$. The coefficient functions $c_{1,0}, \ldots, c_{n-1,0}$ are constructed inductively via the next lemma. This lemma does not determine the coefficients functions $c_{n,q}$ for all $q$. However, if those coefficient functions are known, it can be used to determine all the coefficient functions $c_{p,q}$ with $p \geq n+1$.

\begin{lem}\label{lem:Wm-1Wm}
Let $\Omega$ be a bounded domain in $\mathbb{R}^n$, $n \geq 3$, with smooth boundary $\partial\Omega$. Assume for some $m \geq 1$ that one has found a sequence $\{N_p\}_{1 \leq p \leq m-1}$ of non-negative integers and a sequence $\{c_{p,q}\}_{1 \leq p \leq m-1, 0 \leq q \leq N_p}$  of smooth functions on $\partial\Omega$ such that the function
\[
W_{m-1} = -\ln d + \sum_{p=1}^{m-1} \sum_{q = 0}^{N_p} (c_{p,q} \circ \pi)\, d^p (\ln d)^q
\]
satisfies $\lambda(-S(W_{m-1})) \in \Gamma_k$ near $\partial\Omega$ and $G(W_{m-1}) \in \mathcal{W}^{m-1}$. If $m \neq n$, then there exist a minimal $N_m \geq 0$ and a unique sequence $\{c_{m,q}\}_{0 \leq q \leq N_m}$ such that the function
\[
W_{m} = W_{m-1} + \sum_{q = 0}^{N_m} (c_{m,q} \circ \pi)\, d^m (\ln d)^q =  -\ln d + \sum_{p=1}^{m} \sum_{q = 0}^{N_p} (c_{p,q} \circ \pi)\, d^p (\ln d)^q
\]
satisfies $\lambda(-S(W_{m})) \in \Gamma_k$ near $\partial\Omega$ and $G(W_{m}) \in \mathcal{W}^{m}$. 
\end{lem}

The following remarks are clear from the proof of the lemma.
\begin{rem}
\begin{enumerate}[(i)]
\item When $m \neq n$, the coefficient functions $c_{m,0}, \ldots, c_{m, N_m}$ can be written as some (explicitly computable) polynomials of the principal curvatures $\kappa_i$ of $\partial\Omega$, the coefficient functions $c_{p,q}$ ($1 \leq p \leq m-1$, $0 \leq q \leq N_p$) and their derivatives up to second order.

\item If some function
\[
\tilde W_{m} = W_{m-1} + \sum_{q = 0}^{\tilde N_m} (\tilde c_{m,q} \circ \pi)\, d^m (\ln d)^q 
\]
satisfies $\lambda(-S(\tilde W_{m})) \in \Gamma_k$ near $\partial\Omega$ and $G(\tilde W_{m}) \in \mathcal{W}^{m}$, then $\tilde N_m \geq N_m$, $\tilde c_{m,q} = c_{m,q}$ for $0 \leq q \leq N_m$ and $\tilde c_{m,q} \equiv 0$ for $N_m < q \leq \tilde N_m$.
\end{enumerate}
\end{rem}

\begin{proof}
For convenience, we also write $c_{p,q} = 0$ for $q > N_p$. Let $\hat c_{p,q} = c_{p,q} \circ \pi$. Pick an arbitrary point $\bar x\in \om \setminus \om_{\delta_0} $ and, after a translation and rotation of the coordinate axes, write $\bar x=(\bar x',\bar x_n)=(0',\bar x_n)$, $\bar x_n> 0$ and $\pi(\bar x)=0\in \partial \Omega$ as in the proof of Lemma \ref{lem:W0}. For $N_m$, $c_{m,1}, \ldots, c_{m,N_m}$ to be chosen, we compute at $\bar x$,
\begin{align*}
\partial_i W_m
	&= \begin{cases}
	\underbrace{\partial_i W_{m-1}}_{O(d)} + \sum_{q=0}^{N_m} \partial_i \hat c_{m,q} d^m  (\ln d)^q & \text{ if } i < n,\\[2mm]
	\underbrace{\partial_n W_{m-1}}_{= -\frac{1}{d} + O(1)} 
		+ \sum_{q=0}^{N_m}[m \hat c_{m,q} + (q+1)\hat c_{m,q+1}] d^{m-1}  (\ln d)^q &\text{if } i = n.
	\end{cases}
\end{align*}
and
\begin{align*}
\partial_i\partial_j W_m
	&= \begin{cases}
	\partial_i \partial_j W_{m-1} + \sum_{q=0}^{N_m} \partial_i\partial_j \hat c_{m,q} d^m  (\ln d)^q\\
		\qquad 
		+ \sum_{q=0}^{N_m} [m \hat c_{m,q} + (q+1)\hat c_{m,q+1}] d^{m-1}  (\ln d)^q \partial_i  \partial_j d
		&\text{if }i,j < n,\\[2mm]
	\partial_i \partial_n W_{m-1}
		+ \sum_{q=0}^{N_m} [m \partial_i \hat c_{m,q} + (q+1) \partial_i \hat c_{m,q+1}]  d^{m-1}  (\ln d)^q
		& \text{if } i < j = n,\\[2mm]
	\partial_n^2 W_{m-1} 
		+  \sum_{q=0}^{N_m} \Big[m(m-1)\hat c_{m,q} + (2m - 1)(q+1)\hat c_{m,q+1}\\
			\quad\qquad
			+  (q+1) (q+2)\hat c_{m,q+2}\Big]d^{m-2}  (\ln d)^q
		&\text{if } i = j = n.
\end{cases}
\end{align*}

In the sequel, we shall use $\textrm{err}_m$ to denote some function of the form $\sum_{i=m+1}^\infty \sum_{j=0}^{M_i} (b_{i,j} \circ \pi)\, d^{i} \,(\ln d)^j \in \mathcal{W}^m$ where each coefficient function $b_{i,j}$ are polynomials of the principal curvatures $\kappa_1$, \ldots, $\kappa_{n-1}$, the functions $\hat c_{p,q}$ with $1 \leq p \leq m, 0 \leq q \leq N_p$ and their derivatives up to the second order. 

We have, at $\bar x$,
\begin{align*}
d^2\partial_i W_m \partial_j W_m
	&= \begin{cases}
	d^2\partial_i W_{m-1} \partial_j W_{m-1}	
		 + \textrm{err}_{m} & \text{if } i, j < n,\\[2mm]
	d^2\partial_i W_{m-1} \partial_n W_{m-1}
		+ \textrm{err}_{m} & \text{if } i < n = j,\\[2mm]
	d^2(\partial_n W_{m-1})^2	
		- \sum_{q=0}^{N_m} 2\Big[m \hat c_{m,q} \\
			\qquad+ (q+1)\hat c_{m,q+1}\Big] d^{m}  (\ln d)^q
		+ \textrm{err}_{m}
		&\text{ if } i = j = n,
	\end{cases}\\
d^2|\nabla W_m|^2
	&= d^2|\nabla W_{m-1}|^2
			- \sum_{q=0}^{N_m} 2[m \hat c_{m,q} + (q+1)\hat c_{m,q+1}] d^{m}  (\ln d)^q 
			+\textrm{err}_{m}.
\end{align*}
As $-d^2 S_{ij}(W_m)= d^2 \pa_{ij}^2 W_m-d^2 \pa_i W_m\pa_j W_m+\frac{d^2}{2} |\nabla W_m|^2 \delta_{ij}$, we thus have
\begin{align}
-d^2 S_{ij}(W_m) 
	&= -d^2 S_{ij}(W_{m-1})\nonumber\\
		&\quad + \begin{cases}
		-\sum_{q=0}^{N_m} 	\Big[m \hat c_{m,q} \\
			\qquad + (q+1)\hat c_{m,q+1}\Big] \delta_{ij} d^{m}  (\ln d)^q
		  + \textrm{err}_{m} & \text{if } i,j < n,\\[2mm]
		\textrm{err}_{m} & \text{if } i < n = j,\\[2mm]
		 \sum_{q=0}^{N_m} \Big[m^2\hat c_{m,q} + 2m(q+1)\hat c_{m,q+1}\\
			\qquad
			+  (q+1) (q+2)\hat c_{m,q+2}\Big]d^{m}  (\ln d)^q
			+ \textrm{err}_{m}
			&\text{if } i = j = n.
		\end{cases}
		\label{Eq:SWmRec}
\end{align}

Recall that $-d^2S(W_0) = \frac{1}{2}I + O(d)$. Thus a simple induction using \eqref{Eq:SWmRec} gives that $-d^2S(W_{m-1}) = \frac{1}{2}I + O(d)$, $-d^2S(W_m) = \frac{1}{2}I + O(d)$ and in particular $\lambda(-S(W_m)) \in \Gamma_k$ near $\partial\Omega$. Furthermore, in view of the fact 
\begin{equation} \label{eq:sigma-k-fact}
\sigma_{k}(\lambda(-d^2 S(W_m) ))=(-1)^kd^{2k} \frac{1}{k!} \delta_{j_1\dots j_k}^{i_1\dots i_k} S_{i_1j_1}(W_m) \dots S_{i_kj_k}(W_m)
\end{equation}
where $\delta_{j_1\dots j_k}^{i_1\dots i_k}$ denotes the generalized Kronecker symbol, we have that
\begin{align}
&\sigma_k(\lambda(-d^2 S(W_m)))
	- \sigma_k(\lambda(-d^2 S(W_{m-1})))\nonumber\\
	&\quad= 2^{-(k-1)}\tbinom{n-1}{k-1} \sum_{q=0}^{N_m} 	\Big[ m(m-n+1)\hat c_{m,q}\nonumber \\
		&\qquad+ (2m - n+1)(q+1)\hat c_{m,q+1}
			+  (q+1) (q+2)\hat c_{m,q+2}\Big]d^m  (\ln d)^q
			+ \textrm{err}_m.
	\label{Eq:sigmakWW}
\end{align}
Since
\[
d^{2k} e^{2k W_m} - d^{2k} e^{2k W_{m-1}}
	= 2k \sum_{q=0}^{N_m} \hat c_{m,q} d^m (\ln d)^q + \textrm{err}_m,
\]
we deduce that
\begin{align}
&G(W_m) - G(W_{m-1})
	= 2^{-(k-1)}\tbinom{n-1}{k-1} \sum_{q=0}^{N_m} 	\Big[ (m+1)(m-n)\hat c_{m,q} \nonumber\\
		&\qquad+ (2m - n+1)(q+1)\hat c_{m,q+1}
			+  (q+1) (q+2)\hat c_{m,q+2}\Big]d^m  (\ln d)^q
			+ \textrm{err}_m.
		\label{Eq:GWW}
\end{align}

To proceed, we consider separately the cases $G(W_{m-1}) \in \mathcal{W}^{m-1} \setminus \mathcal{W}^{m}$ and $G(W_{m-1}) \in \mathcal{W}^{m-1}$.

Consider the case $G(W_{m-1}) \in \mathcal{W}^{m-1} \setminus \mathcal{W}^{m}$. Then there exist $\tilde N_{m} \geq 0$ and a sequence $\{a_{m-1,m,q}\}_{0 \leq q \leq \tilde N_{m}}$ with $a_{m-1,m,\tilde N_{m}} \not\equiv 0$ such that, for every $j \geq m$ and $\gamma \in (0,1)$, 
\[
G(W_{m-1}) - \sum_{q=0}^{\tilde N_{m}} (a_{m-1,m,q} \circ \pi) d^m \,(\ln d)^q \in \mathcal{W}^m.
\]
From \eqref{Eq:GWW}, we have that $G(W) \in \mathcal{W}^m$ if and only if $N_m \geq \tilde N_{m}$ and $c_{m,q}$'s satisfy
\begin{multline}
2^{-(k-1)}\tbinom{n-1}{k-1} \Big[ (m+1)(m-n) c_{m,q} 
		+ (2m - n+1)(q+1) c_{m,q+1}\\
		+  (q+1) (q+2) c_{m,q+2}\Big]
		= -a_{m-1,m,q} \quad \text{ for } 0 \leq q \leq N_m.
		\label{Eq:CoefRecForm}
\end{multline}
When $N_m$ takes the minimal values $\tilde N_m$, since $m \neq n$ and $c_{m,N_m+1}= c_{m,N_m+2} = 0$ (by our convention), we may solve \eqref{Eq:CoefRecForm} successively and uniquely for $c_{m,N_m}$, $c_{m, N_m - 1}$, \ldots, and eventually $c_{m,0}$, which concludes the proof in this case.

It remains to consider the case $G(W_{m-1}) \in \mathcal{W}^{m}$. The same argument as above, but simpler, gives that $G(W) \in \mathcal{W}^m$ if and only if $c_{m,q} = 0$ for all $q$. The conclusion then follows with $N_m = 0$ and $c_{m,0} \equiv 0$.
\end{proof}

While Lemma \ref{lem:Wm-1Wm} cannot determine the coefficient functions $c_{n,q}$ for all $q$, a careful inspection of the proof gives that $c_{n,1}$ can be determined this way. We have:

\begin{prop} \label{prop:expansion-1} Let $\Omega$ be a bounded domain in $\mathbb{R}^n$, $n \geq 3$, with smooth boundary $\partial\Omega$. Then there exist $n$ smooth functions $c_{1,0},\ldots, c_{n-1,0}, c_{n,1}$ on $\pa \om$ such that, for any smooth function $\mu$ on $\partial\Omega$, the function
\[
W= -\ln d+\sum_{p=1}^{n-1} (c_{p,0} \circ \pi) d^p + (c_{n,1} \circ \pi)\, d^n\ln d + (\mu \circ \pi) d^n
\]
satisfies $\lambda(-S(W)) \in \Gamma_k$ near $\partial\Omega$ and $G(W) \in \mathcal{W}^n$. Moreover, if $\tilde c_{1,0},\ldots, \tilde c_{n-1,0}, \tilde c_{n,1}, \tilde \mu$ are $n+1$ smooth functions on $\partial\Omega$ such that the function 
\[
\tilde W= -\ln d+\sum_{p=1}^{n-1}( \tilde c_{p,0} \circ \pi) d^p +  (\tilde c_{n,1} \circ \pi)\, d^n\ln d + (\tilde\mu \circ \pi) d^n
\]
satisfies $G(\tilde W) \in \mathcal{W}^n$, then $\tilde c_{1,0} = c_{1,0},\ldots, \tilde c_{n-1,0} = c_{n-1,0}$ and $\tilde c_{n,1} = c_{n,1}$.
\end{prop} 

\begin{rem}
It is seen from the proof that each $c_{p,q}$ in the proposition above other than $c_{n,0}$ is a polynomial in terms of the principal curvatures $\kappa_1, \ldots, \kappa_{n-1}$ of $\partial\Omega$ and their covariant derivatives up to order $2(p-1)$. In particular, 
\[
c_{1,0} =\frac{1}{2(n-1)} (\kappa_1 + \ldots + \kappa_{n-1}). 
\]
\end{rem}

\begin{proof} 
Using Lemma \ref{lem:W0} and applying Lemma \ref{lem:Wm-1Wm} successively $(n-1)$ times with $N_1 = \ldots = N_{m-1} = 0$, we find $(n-1)$ smooth functions $ c_{1,0},\ldots, c_{n-1,0}$ on $\partial\Omega$ such that, for each $1 \leq m \leq n-1$ and with
\[
W_m = -\ln d + \sum_{p=1}^m (c_{p,0} \circ \pi) d^p,
\]
we have $\lambda(-S(W_m)) \in \Gamma_k$ near $\partial\Omega$ and $G(W_m) \in \mathcal{W}^m$. In addition, it is seen from the proof of Lemma \ref{lem:Wm-1Wm} that if $G(\tilde W) \in \mathcal{W}^m$ for some $1 \leq m \leq n-1$, then $\tilde c_{1,0} = c_{1,0}, \ldots, \tilde c_{m,0} = c_{m,0}$. Moreover, so far no logarithmic terms appear in the expansion of $G(W_m)$ near $\partial\Omega$, i.e. $G(W_m) = \sum_{p=m+1}^\infty (a_{p,0} \circ \pi) d^p$. In particular, there exists a smooth function $a_{n,0}$ on $\partial\Omega$ such that
\[
G(W_{n-1}) - a_{n,0} \circ \pi d^n  \in \mathcal{W}^n.
\]
Thus, by formula \eqref{Eq:GWW}, a function of the form
\[
W= W_{n-1} + (c_{n,1} \circ \pi) d^n(\ln d)^q + (\mu \circ \pi) d^n
\]
satisfies $G(W) \in \mathcal{W}^{n}$ if and only if  
\[
c_{n,1} = - \frac{a_{n,0}}{2^{-(k-1)}(n+1) \tbinom{n-1}{k-1}} \text{ and } \mu \text{ is arbitrary}.
\]
On the other hand, by \eqref{Eq:SWmRec} and the fact that $\lambda(-S(W_{n-1})) \in \Gamma_k$ near $\partial\Omega$, we also have that $\lambda(-S(W)) \in \Gamma_k$ near $\partial\Omega$. The conclusion follows.
\end{proof}

\subsection{Step 2: Approximations of $w$ up to order $O(d^n)$}

We next show that the function $W$ constructed in Proposition \ref{prop:expansion-1} is a good approximation for the solution of \eqref{eq:main-2a}--\eqref{eq:main-2b} up to an $O(d^n)$ error at any derivative level. To show that the approximation is good in $C^0$ sense, we use the following lemma on sub- and super-solutions.

\begin{lem} \label{lem:Bar1} Let $\Omega$ be a bounded domain in $\mathbb{R}^n$, $n \geq 3$, with smooth boundary $\partial\Omega$. Let $c_{1,0},\ldots, c_{n-1,0}, c_{n,1}$ and $W$ be given by Proposition \ref{prop:expansion-1} with the choice $\mu = 0$. Then for $\theta\in (n,n+1)$  there exists $0<\delta_1\le \delta_0/2$, depending only on $n,k,\theta$ and $\om$,  such that for any positive  constants $0<\delta\le \delta_1$ and $\beta\ge 1$  with $\beta \delta^n \le 1 $, there hold  
\begin{align} \label{eq:barrier-up}
G(W+ \beta (d^n-d^{\theta})) &< 0 \quad \mbox{in }\om \setminus \om_\delta,\\
 \label{eq:barrier-low}
G(W- \beta (d^n-d^{\theta})) &>0 \quad \mbox{in }\om \setminus \om_\delta.
\end{align} 
\end{lem} 

\begin{proof} Let $W_\beta = W - \beta (d^n-d^{\theta})$. In the proof, all implicit constants in big $O$-terms are independent of $\beta$. Pick an arbitrary point $\bar x\in \om \setminus \om_{\delta} $ and write $\bar x=(\bar x',\bar x_n)=(0',\bar x_n)$, $\bar x_n> 0$ and $\pi(\bar x)=0\in \partial \Omega$ as in the proof of Lemma \ref{lem:W0}. Computing as in the proof of Lemma \ref{lem:Wm-1Wm}, we have, at $\bar x$,
\begin{align*}
\partial_i W_\beta
	&= \begin{cases}\partial_i W = O(d) & \text{ if } i < n,\\[2mm]
	\underbrace{\partial_n W}_{=-\frac{1}{d} + O(1)}
		- n \beta d^{n-1} + \theta\beta d^{\theta-1} &\text{if } i = n.
	\end{cases}
\end{align*}
and
\begin{align*}
\partial_i\partial_j W_\beta
	&= \begin{cases}
	\partial_i \partial_j W \underbrace{- n \beta d^{n-1} \partial_i \partial_j d + \theta \beta d^{\theta-1}\partial_i \partial_j d}_{=  O(|\beta| d^{n-1})}&\text{if }i,j < n,\\[2mm]
	\partial_i \partial_n W& \text{if } i < j = n,\\[2mm]
	\partial_n^2 W
		- n(n-1)\beta d^{n-2} + \theta(\theta-1)\beta d^{\theta-2}.
		&\text{if } i = j = n.
\end{cases}
\end{align*}
Hence, at $\bar x$,
\begin{align*}
d^2\partial_i W_\beta \partial_j W_\beta
	&= \begin{cases}
	d^2\partial_i W \partial_j W& \text{if } i, j < n,\\[2mm]
	d^2\partial_i W \partial_n W
		+ O(|\beta| d^{n+2})& \text{if } i < n = j,\\[2mm]
	d^2(\partial_n W_{m-1})^2	+ 2n\beta d^n -2 \theta \beta d^\theta + O(|\beta| d^{n+1})
		&\text{ if } i = j = n,
	\end{cases}\\
d^2|\nabla W_\beta|^2
	&= d^2|\nabla W|^2  + 2n\beta d^n -2 \theta \beta d^\theta + O(|\beta| d^{n+1}),
\end{align*}
and
\begin{align*}
-d^2 S_{ij}(W_\beta)&=  d^2 \pa_{ij}^2 W_\beta-d^2 \pa_i W_\beta\pa_j W_\beta+\frac{d^2}{2} |\nabla W_\beta|^2 \delta_{ij} \nonumber  \\&=-d^2 S_{ij}(W) + \begin{cases}
n\beta  d^{n}  - \theta \beta  d^{\theta} + O(|\beta|d^{n+1})     &\mbox{if }i=j<n, \\[2mm]
 -n^2 \beta d^n+  \theta^2 \beta d^{\theta}  + O(|\beta|d^{n+1}) & \mbox{if }i=j=n,\\[2mm]
O(|\beta|d^{n+1}) & \mbox{otherwise}. 
\end{cases}
\end{align*}
Using the fact  \eqref{eq:sigma-k-fact} with  $k\ge 2$ and that $-d^2 S(W) = \frac{1}{2}I + O(d)$, we have 
\begin{align*}
\sigma_k (\lda(-d^2 S_{ij}(W_\beta) ))&= \sigma_k (\lda(-d^2 S_{ij}(W) )) 
	 - 2^{-k+1} \tbinom{n-1}{k-1}n \beta  d^{n}  \\
	 &\qquad
	 +2^{-k+1} \tbinom{n-1}{k-1}   \theta ( \theta+1 -n  ) \beta d^{\theta} +O(|\beta| d^{n+1}).   
\end{align*}
Also,
\begin{align*}
d^{2k}e^{2k W_\beta} - d^{2k} e^{2kW}&= 
-2k \beta d^n + 2k\beta d^\theta + O(|\beta| d^{n+1}).
\end{align*}
Therefore,
\begin{align*}
G(W_\beta) - G(W)=  2^{-k+1}\tbinom{n-1}{k-1} (\theta-n)(\theta+1)   \beta d^\theta+ O(|\beta|d^{n+1}). 
\end{align*}
Since $G(W) \in \mathcal{W}^{n}$, we have $G(W) = O(d^{\frac{1}{2}(\theta + n+1)}$ and so \eqref{eq:barrier-low} follows. Reversing the sign of $\beta$ in the above computation, we obtain also \eqref{eq:barrier-up}.
\end{proof}

We now show that $w - W = O(d^n)$ at all derivative level and in particular the regularity of $w$ near the boundary.

\begin{prop} \label{prop:PreOpExp} 
Let $\Omega$ be a bounded domain in $\R^n$, $n \geq 3$, with smooth boundary $\partial \Omega$. Then there exists $0<\delta_1\le \delta_0/2$, depending only on $n,k$ and $\om$, such that the solution of \eqref{eq:main-2a}--\eqref{eq:main-2b} satisfies
\[
 w + \ln d \in C^\infty(\Omega \setminus\Omega_{\delta_1}) \cap C^{n-1,\gamma}(\overline{\Omega \setminus \Omega_{\delta_1}}) \text{ for every } \gamma \in (0,1).
\]
Moreover, with $c_{1,0},\ldots, c_{n-1,0}, c_{n,1}$ given by Proposition \ref{prop:expansion-1} and
\[
W= -\ln d+\sum_{p=1}^{n-1} (c_{p,0} \circ \pi) d^p + (c_{n,1} \circ \pi)\, d^n\ln d,
\]
we have for every $j \geq 0$ that
\begin{equation} \label{eq:u-expan}
\Big|\nabla^j \big(w(x) - W(x)\big)\Big| = O(d(x)^{n-j}) \quad \text{ as } d(x) \rightarrow 0.
\end{equation}
Finally, for any $s, j \geq 0$ and $\va > 0$, we have
\begin{equation}\label{eq:u-expan-b}
|\nabla_T^s  \nabla^j (w - W)| = O(d^{n-\va - j})  \text{ as } d(x) \rightarrow 0.
\end{equation}

\end{prop} 

\begin{proof}

\medskip
\noindent\underline{Step 1:} We prove that
\begin{equation}
|w(x) - W(x)| = O(d(x)^n) \text{ as } d(x) \rightarrow 0.
	\label{eq:uc0exp}
\end{equation}
This gives \eqref{eq:u-expan} for $j = 0$.

Let $\theta=n+1/2$, by Lemma \ref{lem:Bar1}, there exists $0<\tilde \delta_1\le \delta_0/2$, depending only on $n,k$ and $\om$,  such that for any positive  constants $0<\delta\le \tilde\delta_1$ and $A\ge 1$  with $A \delta^n \le 1 $, there hold  
\begin{equation*} 
G(W+ A (d^n-d^{\theta})) <0 \quad \mbox{in }\om \setminus \om_\delta
\end{equation*} 
and 
\begin{equation*}
G(W-A (d^n-d^{\theta})) >0 \quad \mbox{in }\om \setminus \om_\delta.
\end{equation*}

By \eqref{eq:main-2c}, we can find a small $0 < \delta < \min(\tilde\delta_1,1)$ such that $|w - W| < \frac{1}{10}$ in $\Omega \setminus \Omega_\delta$ and $\delta^n - \delta^\theta \geq \frac{1}{2} \delta^n$. We then select $A = \delta^{-n}$ so that $A \geq 1$, $A\delta^n = 1$ and
\[
-A(d^n-d^\theta)\le w-W\le A (d^n-d^{\theta}) \quad \mbox{on }\pa \om_\delta. 
\]
 
Note that for every $\varepsilon >0$, $w+\varepsilon$ and $w-\varepsilon$  are respectively super- and sub-solutions of \eqref{eq:main-2a} in $\Omega$. By \eqref{eq:main-2c}, we have 
\[
W -A(d^n-d^\theta) \le w+\varepsilon , \quad w-\varepsilon \le W+A(d^n-d^\theta) \quad  \mbox{ in } \Omega \setminus \Omega_{\hat \delta(\varepsilon)},
\] 
where $\hat \delta(\varepsilon) \to 0$ as  $\varepsilon \to 0$. By the comparison principle (see \cite[Theorem 1.3]{LiNgWang18-CVPDE}), we have 
\[
W-A(d^n-d^\theta)   \le w+\varepsilon , \quad w-\varepsilon \le W+A(d^n-d^\theta)  \quad \mbox{in  }\Omega_{\hat \delta(\varepsilon)} \setminus \Omega_\delta \text{ and hence in } \Omega \setminus \Omega_\delta. 
\] 
Sending $\varepsilon $ to $0$, we obtain \eqref{eq:uc0exp} and conclude Step 1.

\medskip
\noindent\underline{Step 2:} We prove that $w$ is smooth in $\Omega \setminus \Omega_{\delta_1}$ for some $\delta_1 \in (0,\tilde\delta_1)$. This will be achieved by using \eqref{eq:uc0exp} to recast \eqref{eq:main-1a} near $\partial\Omega$ in a form suitable to apply a $C^{2,\alpha}$ result of Savin \cite{Savin07-CPDE} for small perturbation solutions for elliptic equations.

Consider a point $x_0 \in \partial\Omega$, which is taken without loss of generality as the origin, and set up coordinate axes so that the $x_n$ axis point toward $\Omega$ and the axes $x_1, \ldots, x_{n-1}$ are along direction tangential to $\partial\Omega$. Since $\partial\Omega$ is smooth, there exists $\varrho_0 > 0$ depending only on $\partial\Omega$ such that, after possibly shrinking $\tilde\delta_1$, the wedges 
\[
\Omega_{x_0,\varrho,\delta} := \{x \in \mathbb{R}^n: |x'| < \delta,  \varrho |x'| < x_n < \delta\}
\]
are subsets of $\Omega$ for $\varrho > \varrho_0$ and $0 < \delta < \delta_1$. It suffices to show that, for some $\delta_1 \in (0,\tilde\delta_1)$ independent of $x_0$, $w$ is $C^{2,\alpha}$ regular in $\Omega_{x_0, 2\varrho_0,\delta_1}$, as smoothness follows from elliptic regularity theories.

It is important to observe that the function $h(x) := -\ln x_n$ satisfies $G(h) = 0$ in $\{x_n > 0\}$ and hence in $\Omega_{x_0,\varrho_0,\delta_0}$. Consider the function
\[
\hat w(x) = -h(x) + w(x) = \ln x_n + w(x)
\]
and define 
\[
\hat S(\hat w)(x)
	:= x_n^2 S(w)(x) 
	= x_n^2 S(\hat w)(x)
		- x_n (e_n \otimes \nabla \hat w +\nabla \hat w \otimes  e_n)
		+ x_n e_n \partial_n \hat w I
		- \frac{1}{2} I.
\]
Since $x_n^{2k}[\sigma_k(\lambda(-S(w)) - N_k e^{2kw}]	= 0$  in $\Omega_{x_0,\varrho_0,\delta_0}$ in the viscosity sense, we have
\[
\hat G(\nabla^2 \hat w, \nabla \hat w, \hat w, x)
	:= \sigma_k (\lambda(-\hat S(\hat w)) - N_k e^{2k \hat w} 
		= 0 \text{ in } \Omega_{x_0,\varrho_0,\delta_0} \text{ in the viscosity sense.}
\]
Clearly $\hat G$ is smooth and, as $G(h) = 0$, $0$ is a solution for $\hat G$, i.e. $\hat G(0,0,0,x) = 0$. Note that $\hat G$ has the following scaling property: If we define $\hat w_r(x) = \hat w(x/r)$ for $r > 0$ and $x \in \Omega_{x_0,\varrho_0, r\tilde\delta_1}$, then $\hat S(\hat w_r)(x) = \hat S(\hat w)(x/r)$ and 
\[
\hat G(\nabla^2 \hat w_r(x), \nabla \hat w_r(x), \hat w_r(x), x)
	= \hat G(\nabla^2 \hat w(x/r), \nabla \hat w(x/r), \hat w(x/r), x/r) = 0 \text{ in } \Omega_{x_0,\varrho_0, r\tilde\delta_1}.
\]
In particular, for $r > 2\tilde\delta_1^{-1}$,
\[
\hat G(\nabla^2 \hat w_r(x), \nabla \hat w_r(x), \hat w_r(x), x)
	= 0 \text{ in } \{x \in \Omega_{x_0,\varrho_0, r\tilde\delta_1}: |x_n - 1| < 1/2\}.
\]
Now observe that $\hat G$ is uniformly elliptic near zero, i.e. there exist $\eta_0 > 0$ and $C_0 > 1 $ depending only on $n$ and $k$ such that, for $\|M\| + |p| + |z| < \eta_0$ and $x \in \Omega_{x_0,\varrho_0, r\tilde\delta_1}$ with $r > 2\tilde\delta_1^{-1}$ and $|x_n - 1| < 1/2$,
\[
\frac{1}{C_0} I \leq \left(\frac{\partial \hat G(M,p,z,x)}{\partial M_{ij}}\right) \leq C_0I.
\]
Moreover, by \eqref{eq:uc0exp}, we have $|\hat w(x)| \leq C_1x_n$ in $\Omega_{x_0,\varrho_0, \tilde\delta_1}$ and so
\[
|\hat w_{r}(x)| \leq \frac{C_1}{r} \text{ in } \{x \in \Omega_{x_0,\varrho_0, r\tilde\delta_1}: |x_n - 1| < 1/2\},
\]
for some constant $C_1$ independent of $x_0$ and $r$. Therefore, by a $C^{2,\alpha}$ result of Savin \cite[Theorem 1.3]{Savin07-CPDE} for small perturbation solutions for elliptic equations, for every $\eta_1 \in (0,\eta_0)$, we can select $r_1$ sufficiently large, depending only on $n, k, \eta_1$, $C_0$ and $C_1$, such that
\[
\|\hat w_{r}\|_{C^{2,\alpha}(\{x \in \Omega_{x_0,2\varrho_0, r\tilde\delta_1}: |x_n - 1| < 1/4\})} \leq \eta_1 \text{ for all } r > r_1.
\]
Thus, with $\delta_1 = \tilde\delta_1/r_1$, we have $w \in C^{2,\alpha}_{\rm loc}(\Omega \setminus \Omega_{\delta_1})$. Moreover, 
\begin{equation}
|\nabla^j (w - \ln d)(x)| \leq (1 + O(d(x))\eta_1\,d(x)^{-j} \text{ for all } x \in \Omega \setminus \Omega_{\delta_1}, j = 0, 1, 2,
\label{Eq:w2Savinest}
\end{equation}
which finishes Step 2.

\medskip
\noindent\underline{Step 3:} We prove \eqref{eq:u-expan} for $j \geq 1$. Note that this implies that $w + \ln d \in C^{n-1,\gamma}(\overline{\Omega \setminus \Omega_{\delta_1}})$ for every $\gamma \in (0,1)$.

For a function $v$, we write $F(\nabla^2 v, \nabla v, v) := \sigma_k(\lambda(-S(v))) - N_k e^{2kv}$. Fix an arbitrary ball $B(x_0, 2r_0) \subset \Omega \setminus \Omega_{\delta_1}$. For $x \in B(0,1)$, define
\begin{align*}
w_{x_0} (x) 
	&= \ln (r_0) + w(x_0 + r_0 x) ,\\
W_{x_0} (x) 
	&= \ln (r_0) + W(x_0 + r_0 x).
\end{align*}
Then
\begin{align*}
F(\nabla^2w_{x_0}, \nabla w_{x_0}, w_{x_0}) = 0, \quad \lambda(-S(w_{x_0})) \in \Gamma_k \quad \text{ in } B(0,1),\\
F(\nabla^2 W_{x_0}, \nabla W_{x_0}, W_{x_0}) = f_{x_0},  \quad \lambda(-S(W_{x_0})) \in \Gamma_k \quad \text{ in } B(0,1).
\end{align*}
Subtracting these equations, we obtain
\[
L_{x_0} (w_{x_0} - W_{x_0}) = -f_{x_0}  \text{ in } B(0,1)
\]
where $L_{x_0} = \sum_{i,j=1}^na_{x_0}^{ij}(x)\pa_{ij}^2+\sum_{i=1}^n b_{x_0}^i(x) \pa_i +c_{x_0}(x)$ is a uniform elliptic operator with 
\begin{align*}
a_{x_0}^{ij}
	&=\int_{0}^1 \frac{\pa }{\pa M_{ij}} F(\nabla^2 \xi_{t,x_0}, \nabla \xi_{t,x_0}, \xi_{t,x_0})\,\ud t,\\
b_{x_0}^{i}
	&=\int_{0}^1 \frac{\pa }{\pa p_{i}} F(\nabla^2 \xi_{t,x_0}, \nabla \xi_{t,x_0}, \xi_{t,x_0} )\,\ud t,\\
c_{x_0}
	&=\int_{0}^1 \frac{\pa }{\pa z} F(\nabla^2 \xi_{t,x_0}, \nabla \xi_{t,x_0}, \xi_{t,x_0} )\,\ud t,\\
\xi_{t,x_0} 
	&= t w_{x_0} + (1-t)W_{x_0}.
\end{align*}
Now, by \eqref{Eq:w2Savinest} and the fact that $-d^2S(-\ln d) = \frac{1}{2}I + O(d)$, we have that the ellipticity of $L_{x_0}$ in $B(0,1/2)$ is uniform with respect to $x_0 \in \Omega \setminus \Omega_{\delta_1}$. Moreover, by \eqref{eq:uc0exp},
\[
|w_{x_0} - W_{x_0}| \leq Cr_0^n \text{ in } B(0,1),
\]
and, since $G(W) \in \mathcal{W}^n$, $f_{x_0}$ satisfies
\begin{equation}
|\nabla^j f_{x_0}|\leq C_{j} r_0^{n} \text{ in } B(0,1/2)
	\label{Eq:04XII21-M1}
\end{equation}
for some constants $C$ and $C_{j}$ which are independent of $x_0 \in \Omega \setminus \Omega_{\delta_1}$. Therefore, by Schauder's estimates, we have for $j \geq 1$ that
\[
|\nabla^j(w_{x_0} - W_{x_0})| \leq C_j'r_0^n \text{ in } B(0,1/4).
\]
where the constants $C_j'$ are also independent of $x_0$. Returning to $w$, we obtain that
\[
|\nabla^j(w - W)| = O(d^{n-j}) \text{ in } \Omega \setminus \Omega_{\delta_1},
\]
which proves \eqref{eq:u-expan}. Clearly, this implies that $w - W \in C^{n-1,1}(\overline{\Omega \setminus \Omega_{\delta_1}})$ and so $w + \ln d \in C^{n-1,\gamma}(\overline{\Omega \setminus \Omega_{\delta_1}})$ for every $\gamma \in (0,1)$.

\medskip
\noindent\underline{Step 4:} We prove \eqref{eq:u-expan-b}.

Consider first the case $s = 1$ and $j=0$, namely, we show that $|\partial_{Y_1} (w - W)| = O(d^{n-\va})$ for any $\va \in (0,1)$ and any smooth vector field $Y_1$ with $Y_1 \cdot \nabla d = 0$. Since the space of vector fields orthogonal to $\nabla d$ is generated by those which are orthogonal to and commute with $\nabla d$, we may assume that $[Y_1, \nabla d] = 0$ in addition to $Y_1 \cdot \nabla d = 0$.
We write
\begin{equation}
-G(W) = G(w) - G(W) = d^2 \sum_{i,j=1}^n a_{ij} \partial_{ij} (w - W) + d \sum_{i=1}^n b_i \partial_i(w - W) + c(w - W)
	\label{Eq:04XII21-E1}
\end{equation}
where
\begin{align*}
a_{ij} 
	&= \int_0^1 \frac{\partial\sigma_k}{\partial S_{ij}}(-d^2 S(tw + (1-t)W))\,\ud t,\\
b_{i} 
	&= - 2\sum_{j=1}^n \int_0^1 \frac{\partial\sigma_k}{\partial S_{ij}}(-d^2 S(tw + (1-t)W)) d\partial_j(tw + (1-t)W) \,\ud t\\
		&\quad + \int_0^1 \sum_{p,q=1}^n \frac{\partial\sigma_k}{\partial S_{pq}}(-d^2 S(tw + (1-t)W)) d\partial_i(tw + (1-t)W) \,\ud t,\\
c
	&= - 2k N_k d^{2k} \int_0^1 e^{2k(tw + (1-t)W)} \,\ud t.
\end{align*}
Recall that, by \eqref{Eq:SWmRec}, $-d^2S(W)$ takes the form
\[
-d^2S_{ij}(W) = \frac{1}{2}\delta_{ij} + \sum_{p=1}^{n-1} s_{ij,p} \circ \pi d^p + \textrm{err}_{n-1}
\]
where $s_{ij,p} \in C^\infty(\partial\Omega)$ and where $\textrm{err}_{n-1} \in \mathcal{W}^{n-1}$. This together with \eqref{eq:u-expan} and the fact that $[Y_1, \nabla d] = 0$ implies that
\begin{align*}
&a_{ij} 
	= 2^{-(k-1)}\tbinom{n-1}{k-1}(\delta_{ij} + O(d)),\\
&b_i 
	= - 2^{-(k-1)}\tbinom{n-1}{k-1} (n-2) (\partial_i d  + O(d)), \\
&c 
	= 2^{-(k-1)}\tbinom{n-1}{k-1} n (1 + O(d)),\\
&|\partial_{Y_1} a_{ij}| + |[Y_1, \sum_i b_i \partial_i]| + |\partial_{Y_1} c| 
	= O(d).
\end{align*}
Also, since $G(W) \in \mathcal{W}^{n}$, $|\partial_{Y_1} G(W)| = O(d^{n+\va})$. Therefore, by differentiating \eqref{Eq:04XII21-E1} and using once again \eqref{eq:u-expan}, we see that the function $\chi := \partial_{Y_1} (w - W)$ satisfies
\[
L\chi := d^2 \Delta \chi - (n-2) d \nabla d \cdot \nabla \chi
		 - n  \chi = O(d^{n+1-\va}) \text{ and } \chi = O(d^{n-1}).
\]
By a straightforward computation and after possibly shrinking $\delta_1$, we have for small $\va > 0$ that $L(d^{n-2}) \leq 0$ and $L(d^{n-1+\gamma}) \leq - \frac{1}{C} d^{n+\gamma}$ in $\Omega \setminus \Omega_{\delta_1}$. A simple application of the maximum principle thus gives that $|\chi| = O(d^{n-1+\gamma})$. This proves \eqref{eq:u-expan-b} for $s = 1$ and $j = 0$.

To proceed, note that the fact that $G(W) \in \mathcal{W}^n$ in fact gives an estimate stronger than \eqref{Eq:04XII21-M1}, namely
\[
|\nabla_T^s \nabla^j  f_{x_0}| \leq C_{j,s} r_0^{n+s}, \quad \text{ for } s, j \geq 0,
\]
where, by a slight abuse of notation, $\nabla_T$ now denotes the gradient along directions orthogonal to $\nabla d(x_0 + K_0^{-1} r_0 \cdot)$. Also, estimate \eqref{eq:u-expan-b} for $s = 1$ and $j = 0$ gives
\[
|\nabla_T (w_{x_0} - W_{x_0})| \leq Cr_0^{n+1-\va}.
\]
Therefore, the argument in Step 3 in fact gives
\[
|\nabla_T \nabla^j  (w_{x_0} - W_{x_0})| \leq C_j'' r_0^{n+1-\va},
\]
which up on returning to $w$ gives \eqref{eq:u-expan-b} for $s = 1$ and all $j \geq 0$. We may then repeatedly differentiate \eqref{Eq:04XII21-E1} and apply the above argument to obtain \eqref{eq:u-expan-b} for all $s \geq 1, j \geq 0$. 
\end{proof}

\subsection{Step 3: The coefficient function $c_{n,0}$ and the rest of the proof}

The next step in the proof of Theorem \ref{thm:mainX} is to find $W_n = W + (c_{n,0} \circ \pi)d^n$ where $W$ is given by Proposition \ref{prop:PreOpExp} such that $w - W_n = o(d^n)$ near $\partial\Omega$. Once this is done, $W_n$ satisfies the condition of Lemma  \ref{lem:Wm-1Wm} with $m = n+1$ (by Proposition \ref{prop:expansion-1}), and we can resume the application of Lemma \ref{lem:Wm-1Wm} to construct higher order approximations $W_{n+1}, W_{n+2}, \ldots$ of $w$. We introduce some notations. For small $\delta$, set $\pi_\delta = \pi|_{\partial\Omega_\delta}:\partial\Omega_\delta  \rightarrow \partial\Omega$ and let $\pi_\delta^{-1}: \partial\Omega  \rightarrow \partial\Omega_\delta$ be its inverse. Define 
\begin{equation}
c_{\delta} = \frac{1}{\delta^n} (w - W)\big|_{\partial\Omega_\delta} \circ \pi_\delta^{-1}: \partial\Omega \rightarrow \mathbb{R}.
	\label{Eq:cdeltaDef}
\end{equation}
By estimate \eqref{eq:u-expan} in Proposition \ref{prop:PreOpExp}, the family $\{c_\delta\}$ is bounded in $C^0(\partial\Omega)$ as $\delta \rightarrow 0$. However, because of the loss of $d^{-\varepsilon}$ in estimate \eqref{eq:u-expan-b} we do not know yet the boundedness of $\{c_\delta\}$ in stronger norms. This is addressed in the next lemma.

\begin{lem}\label{lem:cn0Birth}
Let $\Omega$ be a bounded domain in $\mathbb{R}^n$, $n \geq 3$, with smooth boundary $\partial\Omega$. Let $c_1 = c_{1,0},\ldots, c_{n-1} = c_{n-1,0}, c_{n,1}$ and $W$ be given by Proposition \ref{prop:PreOpExp}, and let $c_\delta$ be defined by \eqref{Eq:cdeltaDef}. Then there exists $c_{n,0} \in C^\infty(\partial\Omega)$ such that $c_\delta$ converges to $c_{n,0}$ in $C^j(\partial\Omega)$ for all $j \geq 0$ as $\delta \rightarrow 0$, and, for any $\gamma \in (0,1)$ and $s \geq 0$,
\begin{equation}
|\nabla_T^s (c_\delta - c_{n,0})|  = O(\delta^\gamma) \text{ for small } \delta > 0.
	\label{Eq:cdrate}
\end{equation}
\end{lem}

\begin{proof} 
Fix some $\va \in (0,1)$. Observe that if $Y_1, \ldots, Y_s$ are vector fields which are orthogonal to and commute with $\nabla d$, we have as in Step 4 of the proof of Proposition \ref{prop:PreOpExp} that the function $\varphi := \partial_{Y_1} \cdots \partial_{Y_s}(w - W)$ satisfies
\[
d^2 \sum_{i,j=1}^n(\delta_{ij} + O(d)) \partial_{ij} \varphi - (n-2) d \sum_{i=1}^n(\partial_i d + O(d)) \partial_i \varphi - n (1 + O(d)) \varphi = O(d^{n+1-\va}).
\]
In the coordinate system in which \eqref{eq:distance} holds, along the $x_n$-axis, we have by \eqref{eq:u-expan-b} that $|\nabla_{x'}^\ell \varphi_{m_0 + 1}(0,x_n)| = O(d^{n-\va})$ for $\ell \geq 0$. Hence, by \eqref{Eq:04XII21-E1}, along the $x_n$-axis, $\varphi = w - W$ satisfies the ODE
\[
x_n^2 \partial_n^2 \varphi(0,x_n) - (n-2) x_n \partial_n \varphi (0,x_n)
		 - n \varphi = O(x_n^{n+1-\va}).
\]
Also by \eqref{eq:u-expan}, we have $|\partial_n^\ell \varphi(0,x_n)| = O(x_n^{n-\ell})$ for $\ell \geq 0$. Thus, introducing $\xi(x_n) := x_n^{-n} \varphi(0,x_n)$, we have
\[
(x_n^{n+2} \xi')' = O(x_n^{n+1-\va}) \quad\text{ and } \quad |\xi| + x_n\,|\xi'| = O(1).
\]
Integrating, we obtain $\xi'(x_n) = O(x_n^{-\va})$, which gives
\[
|\partial_{Y_1} \cdots \partial_{Y_s} c_{x_n}(0) - \partial_{Y_1} \cdots \partial_{Y_s} c_{\tilde x_n}(0)| = |\xi(x_n) - \xi(\tilde x_n)| \leq C |x_n^{1-\va} - \tilde x_n^{1-\va}|.
\]
The conclusion follows.
\end{proof}

\begin{proof}[Proof of Theorem \ref{thm:mainX}] Let $c_1 = c_{1,0}, \ldots, c_{n-1} = c_{n-1,0}, c_{n,1}$ and $W$ be given by Proposition \ref{prop:PreOpExp}  and $c_{n,0}$ be given by Lemma \ref{lem:cn0Birth}. Fix some $\gamma \in (0,1)$. By Proposition \ref{prop:PreOpExp} there exists $\delta_1 > 0$ such that $w + \ln d \in C^\infty(\Omega \setminus \Omega_{\delta_1}) \cap C^{n-1,\gamma}(\overline{\Omega \setminus \Omega_{\delta_1}})$. Let 
\[
W_n = W + c_{n,0}d^n.
\]
By Proposition \ref{prop:expansion-1}, we have that $\lambda(-S(W_n)) \in \Gamma_k$ near $\partial\Omega$ and $G(W_n) \in \mathcal{W}^n$. Using Lemma \ref{lem:Wm-1Wm}, we can find a sequence $\{N_p\}_{p \geq n+1}$ and coefficients $\{c_{p,q}\}_{p \geq n+1, 0 \leq q \leq N_p}$ such that the functions
\[
W_m = W_n + \sum_{p=n+1}^m \sum_{q=0}^{N_p} c_{p,q} \circ \pi d^p (\ln d)^q, \quad m \geq n +1,
\]
satisfy $\lambda(-S(W_m)) \in \Gamma_k$ near $\partial\Omega$ and $G(W_m) \in \mathcal{W}^{m}$. Note also that, if $c_{n,1}$ vanishes in some $\partial\Omega \cap B(x_0,r_0)$, then it can be seen from the proof of Lemma \ref{lem:Wm-1Wm} that, in the expansion of $G(W_n)$, all coefficients carrying a non-trivial power of $\ln d$ actually vanish in $\partial\Omega \cap B(x_0,r_0)$, and so \eqref{Eq:CoefRecForm} implies that $c_{n+1,1}$, \ldots, $c_{n+1,N_{n+1}}$ also vanish in $\partial\Omega \cap B(x_0,r_0)$. Repeating this arguments shows that $c_{p,q} = 0$ in $\partial\Omega \cap B(x_0,r_0)$ for $p \geq n, q \geq 1$.

It remains to prove \eqref{Eq:YYwWmEst}. By estimate \eqref{Eq:cdrate}, for any smooth vector fields $s \geq 0$,
\[
|\nabla_T^s(w - W_n)| = O(d^{n+\gamma}).
\]
We may then argue as in Step 3 of the proof of Proposition \ref{prop:PreOpExp} to obtain for any $s, j \geq 0$ that
\[
|\nabla_T^s  \nabla^j (w - W_n)| = O(d^{n-j+\gamma}) .
\]
We thus have that \eqref{Eq:YYwWmEst} holds for $m = n$ and hence for all $m \leq n$. Suppose by induction that \eqref{Eq:YYwWmEst} has been established for some $m = m_0 \geq n$. Let us prove \eqref{Eq:YYwWmEst} for $m = m_0 + 1$. As in Step 4 of the proof of Proposition \ref{prop:PreOpExp}, if $Y_1, \ldots, Y_s$ are vector fields which are orthogonal to and commute with $\nabla d$, we have that $\varphi_{m_0+1} := \partial_{Y_1} \cdots \partial_{Y_s} (w - W_{m_0+1})$ satisfies
\begin{multline*}
d^2 \sum_{i,j=1}^n(\delta_{ij} + O(d)) \partial_{ij} \varphi_{m_0+1} - (n-2) d \sum_{i=1}^n(\partial_i d + O(d)) \partial_i \varphi_{m_0+1}\\
	 - n (1 + O(d)) \varphi_{m_0+1} = O(d^{m_0+1+\gamma}).
\end{multline*}
We then use the argument in the proof of Lemma \ref{lem:cn0Birth}. In the coordinate system in which \eqref{eq:distance} holds, along the $x_n$-axis, we have by \eqref{Eq:YYwWmEst} for $m = m_0$ that $|\nabla_{x'}^\ell \varphi_{m_0 + 1}(0,x_n)| = O(d^{m_0+\gamma})$ for $\ell \geq 0$, and so the above PDE reduces to an ODE:
\[
x_n^2 \partial_n^2 \varphi_{m_0+1}(0,x_n) - (n-2) x_n \partial_n \varphi_{m_0+1} (0,x_n)
		 - n \varphi_{m_0+1}  = O(d^{m_0+1+\gamma}).
\]
Also by \eqref{Eq:YYwWmEst} for $m = m_0$, we have $|\partial_n^\ell \varphi_{m_0 + 1}(0,x_n)| = O(d^{m_0-\ell+\gamma})$ for $\ell \geq 0$. Thus, introducing $\xi_{m_0+1}(x_n) := x_n^{-n} \varphi_{m_0+1}(0,x_n)$, we have
\[
(x_n^{n+2} \xi_{m_0+1}')' = O(d^{m_0+1+\gamma}) \quad\text{ and } \quad |\xi_{m_0+1}| + x_n\,|\xi_{m_0+1}'| = O(d^{m_0 - n +\gamma}).
\]
Integrating, we obtain $\xi_{m_0+1} = O(d^{m_0 + 1-n+\gamma})$, which gives $\varphi_{m_0+1}(0,x_n) = O(d^{m_0+1+\gamma})$. We may then argue as in Step 3 of the proof of Proposition \ref{prop:PreOpExp} to obtain for any $j \geq 0$ that
\[
|\partial_{Y_1} \cdots \partial_{Y_s} \nabla^j (w - W_{m_0+1})| = O(d^{m_0+1-j+\gamma}),
\]
which gives \eqref{Eq:YYwWmEst} for $m = m_0+1$. The proof is complete.
\end{proof}

\begin{proof}[Proof of the first statement of Theorem \ref{thm:mainfG}] 
The proof is essentially the same as that of Theorem \ref{thm:mainX} except for a few changes which we list below.
\begin{enumerate}[1.]
\item By a rescaling, we may assume that $f(1/2, \ldots, 1/2) = 1$. The expression of $G$ is modified to
\[
G(v) = (f(\lambda(S(-v))) - e^{2v})d^2.
\]

\item In Lemma \ref{lem:W0}, the conclusion is that $G(W_0) \in \mathcal{W}^0$ and the polyhomogeneous expansion of $G(W_0)$ has no logarithmic term, i.e. $G(W_0) = \sum_{p=0}^\infty (a_{p,0} \circ \pi) d^p$ in the sense of Definition \ref{def:polyhom} (instead of in the sense of a convergent power series). 

\item In the proof of Lemma \ref{lem:Wm-1Wm}, we need to estimate $f(\lambda(-d^2 S(W_m))) - f(\lambda(-d^2 S(W_{m-1})))$ differently since we cannot use \eqref{eq:sigma-k-fact}. Since $f(1/2, \ldots, 1/2) = 1$, $f$ is symmetric and homogeneous of degree one (see \eqref{fG3}), we have $\nabla f(1/2, \ldots, 1/2) = (2/n, \ldots, 2/n)$. Moreover, since $f$ is smooth near $\lambda = (1/2, \ldots, 1/2)$, we have for any $t \geq 1$ that
\begin{align}
f(\lambda(M))
	 &= 1 + \frac{2}{n} \sum_{1 \leq i \leq n} (M_{ii} - 1/2)\nonumber\\
		 &\quad+ \sum_{2 \leq |\alpha| \leq t} c_{\alpha} (M - \frac{1}{2}I)^\alpha + O(|M - \frac{1}{2}I|^{t+1})
	\label{Eq:24II22-1}
\end{align}
for symmetric $n\times n$ matrices $M$ close to $\frac{1}{2}I$, where $\alpha = (\alpha_{ij})_{1 \leq i,j\leq n}$ denotes a multi-index with non-negative integer entries, $|\alpha| := \sum_{1 \leq i,j\leq n} \alpha_{ij}$, $(M - \frac{1}{2}I)^\alpha := \prod_{1 \leq i,j\leq n} (M_{ij} - 1/2\delta_{ij})^{\alpha_{ij}}$, and the coefficients $c_\alpha$ are independent of $t$. Recalling that $-d^2S(W_{m-1})) = \frac{1}{2}I + O(d)$ and $-d^2S(W_{m})) = \frac{1}{2}I + O(d)$ near $\partial\Omega$, we deduce from \eqref{Eq:24II22-1} that
\begin{align}
&f(\lambda(-d^2 S(W_m)))
	- f(\lambda(-d^2 S(W_{m-1})))\nonumber\\
	&~ = \frac{2}{n}  \sum_{1 \leq i \leq n} [-d^2 S_{ii}(W_{m})  + d^2 S_{ii}(W_{m-1})]\nonumber\\
		&\quad + 
		\sum_{2 \leq |\alpha| \leq m} c_{\alpha} \Big\{ [-d^2 S(W_m) - \frac{1}{2}I]^{\alpha} - [-d^2 S(W_{m-1}) - \frac{1}{2}I]^{\alpha}\Big\}
			+ \textrm{err}_m.
		\label{Eq:24II22-2}
\end{align}
Observe from \eqref{Eq:SWmRec} that, for $|\alpha| \geq 2$ and $1 \leq i,j \leq n$,
\[
[-d^2 S_{ij}(W_{m}) - \frac{1}{2}\delta_{ij}]^\alpha -  [-d^2 S_{ij}(W_{m-1}) - \frac{1}{2}\delta_{ij}]^\alpha  = \textrm{err}_m.
\]
Thus all terms in the last line of \eqref{Eq:24II22-2} is $\textrm{err}_m$. Using this as well as \eqref{Eq:SWmRec} in \eqref{Eq:24II22-2} we obtain
\begin{align*}
&f(\lambda(-d^2 S(W_m)))
	- f(\lambda(-d^2 S(W_{m-1})))\\
	&\quad= \frac{2}{n}   \sum_{1 \leq i \leq n} [-d^2 S_{ii}(W_{m})  + d^2 S_{ii}(W_{m-1})]
			+ \textrm{err}_m\\
	&\quad= \frac{2}{n}  \sum_{q=0}^{N_m} 	\Big[ m(m-n+1)\hat c_{m,q}\nonumber \\
		&\qquad+ (2m - n+1)(q+1)\hat c_{m,q+1}
			+  (q+1) (q+2)\hat c_{m,q+2}\Big]d^m  (\ln d)^q
			+ \textrm{err}_m,
\end{align*}
which resembles estimate \eqref{Eq:sigmakWW} for $\sigma_k(\lambda(-d^2 S(W_m))) - \sigma_k(\lambda(-d^2 S(W_{m-1})))$.

\item In the proof of Lemma \ref{lem:Bar1}, we follow the argument in the previous point to estimate $f(\lambda(-d^2 S(W_\beta)))	- f(\lambda(-d^2 S(W)))$.

\item In Step 2 of the proof of Proposition \ref{prop:PreOpExp}, the expression of $\hat G$ is modified to
\[
\hat G(\nabla^2 \hat w, \nabla \hat w, \hat w, x)
	:= f (\lambda(-\hat S(\hat w)) - e^{2\hat w} 
	= x_n^2\big[f(\lambda(-x_n^2 S(w)) - e^{2w}\big].
\]

\item In Step 3 of the proof of Proposition \ref{prop:PreOpExp}, the expression of $F$ is modified to
\[
F(\nabla^2 v, \nabla v, v) := f(\lambda(-S(v))) - e^{2v}.
\]
\end{enumerate}
\end{proof}

\section{Non-differentiability: Theorem \ref{thm:MX} and its consequences}  \label{Sec:MainProof}

Let us first assume Theorem \ref{thm:MX} and give the

\begin{proof}[Proof of Theorem \ref{thm:main} and the second statement of Theorem \ref{thm:mainfG}]
We will only prove Theorem \ref{thm:main}. The proof of the second statement of Theorem \ref{thm:mainfG} is the same.

We write $\Omega = \Omega_\ell \setminus (\Omega_1 \cup \ldots \cup \Omega_{\ell-1})$ and  $\partial \Omega= \partial\Omega_1 \cup \dots \cup \partial \Omega_\ell$ where $\ell\ge 2$, the domains $\Omega_i$'s are bounded and the sets $\partial\Omega_i$'s are disjoint connected components of $\partial \Omega$. For small $\delta > 0$ we write 
\begin{align*}
\om_\delta := \{y\in \om: d_{\mathring{g}}(y, \pa \om)>\delta\} = \Omega_{\ell,\delta} \setminus (\Omega_1^\delta \cup \ldots \cup \Omega_{\ell - 1}^\delta)
\end{align*}
where $\Omega_i \subset \Omega_i^\delta$ for $1 \leq i \leq \ell - 1$, $\Omega_{\ell,\delta} \subset \Omega_\ell$, and the sets $\partial \Omega_1^\delta, \ldots, \partial\Omega_{\ell-1}^\delta, \partial \Omega_{\ell,\delta}$ are disjoint, closed smooth hypersurfaces in $\om$. (Recall that $\mathring{g}$ denotes the Euclidean metric.) See Figure \ref{FigOmd}.

\begin{figure}[h]
\caption{The sets $\Omega$ and $\Omega_\delta$.}\label{FigOmd}
\begin{center}
\begin{tikzpicture}
\draw (0,0) ellipse (4 and 2);
\draw[dashed] (0,0) ellipse (3.8 and 1.8);
\draw (2.8, 1.8) node {$\Omega_\ell$};
\draw (2, 1.2) node {$\Omega_{\ell,\delta}$};
\draw (-1.5,0) ellipse (.5 and 1);
\draw[dashed] (-1.5,0) ellipse (.7 and 1.2);
\draw (-1.6,0) node {$\Omega_1$};
\draw (-2.5,0) node {$\Omega_1^\delta$};

\draw (1.5,0) ellipse (1 and .5);

\draw[dashed] (1.5,0) ellipse (1.2 and .7);
\draw (1.5,-.2) node {$\Omega_2$};
\draw (1.5,-1) node {$\Omega_2^\delta$};

\end{tikzpicture}
\end{center}
\end{figure}
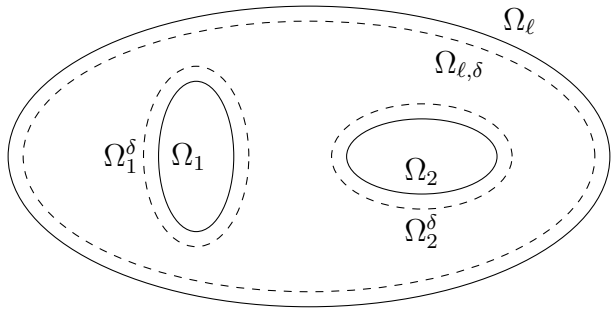

Note that, since $u \in C^1(\Omega)$, if $\Sigma$ is a smooth surface in $\Omega$, then the mean curvature $H_\Sigma$ of $\Sigma$ (in a specified normal direction) with respect to $g$ is well-defined. To dispel confusion, in our notation, the mean curvature is the trace of the second fundamental form. Moreover, if we denote by $\mathring{H}_\Sigma$ the mean curvature of $\Sigma$ with respect to the Euclidean metric $\mathring{g}$, then 
\begin{equation}
H_\Sigma = u^{-\frac{2}{n-2}}\Big(-\frac{2(n-1)}{n-2}\partial_\nu \ln u  + \mathring{H}_\Sigma \Big).
	\label{Eq:06IV20-X1}
\end{equation}
where $\nu$ is the $\mathring{g}$-unit normal to $\Sigma$ along the specified normal direction of $\Sigma$.

Since $u$ is the solution of \eqref{eq:main-1a}-\eqref{eq:main-1b}, by Theorem \ref{thm:mainX}, there exists $\delta_0 > 0$ such that $u \in C^\infty(\Omega \setminus \Omega_{\delta_0})$ and
\begin{equation}
H_{\partial\Omega_\delta} > 0 \text{ for all }0 < \delta < \delta_0,
	\label{Eq:MCBarrier}
\end{equation}
where the mean curvature is computed with respect to the normal pointing towards $\Omega_\delta$. Using \eqref{Eq:MCBarrier} and applying Theorem \ref{thm:MX} to any $\Omega_\delta$ in place of $\Omega$ with some small $\delta > 0$, the result follows immediately.
\end{proof}

The rest of the section is devoted to the proof of Theorem \ref{thm:MX}. The plan is as follows. In Subsection \ref{SSec:3.1}, we set up an obstacle problem to look for a set $\overline{\Omega_1} \subset E \subset \overline{\Omega} \cup \overline{\Omega_1}$ which minimizes the perimeter function with respect to $g$. We show that either a component of $\partial\Omega$ has zero mean curvature or $\Sigma = \partial E$ detaches from $\partial\Omega$, in which case $\Sigma$ is a $C^1$ hypersurface away from a singular set $\mathcal{S} \subset \Sigma$ of Hausdorff dimension at most $n - 8$ and $\Sigma \setminus \mathcal{S}$ has zero mean curvature (see Proposition \ref{prop:Conc}). In Subsection \ref{SSec:3.2}, we then extend an argument made in \cite{LiNg21-JMS} using an obstruction for the existence of a minimal hypersurface for metrics whose Schouten tensors belong to the negative $\bar\Gamma_2$ cone in a smooth context to the current situation (see Lemma \ref{Lem:12XI21MC}) to deduce a contradiction and conclude the proof.

\subsection{Existence and regularity of a minimal hypersurface}\label{SSec:3.1}

Let us recall the notion of perimeter with respect to $g = g_u = u^{\frac{4}{n-2}}\mathring{g}$ when $u > 0$ and $u^{\frac{2n}{n-2}} \in W^{1,1}(\Omega_\ell)$, where throughout the section, unless otherwise stated, all Sobolev spaces are defined using the measure given by the Euclidean metric. For an open set $A \subset \Omega_\ell$, let $\mathcal{X}(A)$ be the set of smooth vector fields compactly supported in $A$. Recall that, for a measurable set $E \subset \Omega_\ell$ and an open set $A \subset \Omega_\ell$, the perimeter $Per_g(E,A)$ of $E$ in $A$ with respect to the metric $g$ is defined as the total variation of the distributional gradient $\nabla_g I_E$  of the characteristic function $I_E$ of $E$ in $A$, i.e. 
\[
Per_g(E,A)  := |\nabla_g I_E|(A)\\
	=  \sup \Big\{ \int_{A} I_E \,\mathrm{div}_{g}  w\, \ud vol_{g}:  w\in \mathcal{X} (A), |w|_{g}\le 1 \text{ in } A\Big\} .
\]
When it is clear from the context, we simply write $Per(E,A)$ instead of $Per_g(E,A)$. If $Per(E,\Omega_\ell) < \infty$ and $\partial E$ is $C^1$ away from a set of zero $(n-1)$-dimensional Hausdorff measure (defined using the metric $g$), then 
\begin{equation}
Per_g(E,A) = \mathcal{H}^{n-1}_g(\partial E \cap A).
	\label{Eq:Per=H}
\end{equation}
We also have the general inequality
\begin{equation}
	Per_g(E_1 \cup E_2, A) + Per_g(E_1 \cap E_2, A) \leq Per_g(E_1,A) + Per_g(E_2,A).
	\label{Eq:PerInOut}
\end{equation}
We refer readers to e.g. \cite{AmbrosioFuscoPallara, Giusti-MinSurf, Maggi12-Book} for more details.

As in the Euclidean case, it is an easy consequence of Lebesgue's dominated convergence theorem that the perimeter function is lower semi-continuous, namely if $I_{E_j} \rightarrow I_E$ a.e. in $\Omega_\ell$ then
\[
Per(E,\Omega_\ell) \leq \liminf_{j \rightarrow \infty} Per(E_j,\Omega_\ell).
\]

The following lemma gives a relation between the perimeter functions defined using two conformal metrics.

\begin{lem}\label{Lem:PerRel}
Let $B_r(x_0)$ be a Euclidean ball and $g = u^{\frac{4}{n-2}} \mathring{g}$ with $\ln u \in L^\infty(B_r(x_0)) \cap W^{1,1}(B_{r}(x_0))$. For any measurable set $E$, we have
\begin{align*}
Per_{\mathring{g}}(E, B_r(x_0)) \leq \|u^{-1}\|_{L^\infty(B_r(x_0))}^{\frac{2(n-1)}{n-2}}  Per_g(E, B_r(x_0)),\\
Per_g(E, B_r(x_0)) \leq \|u\|_{L^\infty(B_r(x_0))}^{\frac{2(n-1)}{n-2}}   Per_{\mathring{g}}(E, B_r(x_0)). 
\end{align*}
\end{lem}

\begin{proof}
We will only prove the first inequality. The proof of the second inequality is the same. For any $w \in \mathcal{X}(B_r(x_0))$ with $|w|_{\mathring{g}} \leq 1$ in $B_r(x_0)$, we have
\begin{align*}
\int_{E \cap B_r(x_0)} \mathrm{div}_{\mathring{g}} w\,\ud vol_{\mathring{g}} 
	&= \int_{E \cap B_r(x_0)} \mathrm{div}_{g} \big(u^{-\frac{2n}{n-2}}w\big)\, \,\ud vol_{g}\\
	&\leq \|u^{-1}\|_{L^\infty(B_r(x_0))}^{\frac{2(n-1)}{n-2}} Per_g(E,B_r(x_0)),
\end{align*}
where we have used $|u^{-\frac{2n}{n-2}}w|_g = u^{-\frac{2(n-1)}{n-2}} |w|_{\mathring{g}} \leq \|u^{-1}\|_{L^\infty(B_r(x_0))}^{\frac{2(n-1)}{n-2}}$. Taking the supremum on the left hand side, we conclude the proof.
\end{proof}

Consider the obstacle problem
\begin{equation}\label{Eq:ObsProb}
J= \inf \left\{Per_g(E,\Omega_\ell) : \overline{\Omega_1} \subset E \subset  \overline{\Omega} \cup \overline{\Omega_1} \right \}. 
\end{equation}

\begin{lem} \label{lem:ex} 
Let $\Omega = \Omega_\ell \setminus (\Omega_1 \cup \ldots \cup \Omega_{\ell-1})$ be a bounded domain in $\R^n$, $n \geq 3$, with smooth and disconnected boundary $\partial \Omega = \partial\Omega_1 \cup \ldots \cup \partial \Omega_\ell$, $\ell \geq 2$. Let $\ln u \in L^\infty({\Omega_\ell}) \cap W^{1,1}(\Omega_\ell)$. Then the minimization problem $J$ in \eqref{Eq:ObsProb} is attained by some set $E$. 
\end{lem}

\begin{proof}
The proof is standard and is included for completeness. Let $E_j$ be a minimizing sequence so that $Per(E_j,\Omega_\ell) \rightarrow J$. Then $\{I_{E_j}\}$ is a bounded sequence in $BV(\Omega_\ell)$. By the compactness theorem of BV functions, we may assume, after extracting a subsequence, that $I_{E_j}$ converges in $L^1(\Omega_\ell)$ and almost everywhere to some $BV$ function, which takes value in $\{0,1\}$ and so is the characteristic function of some measurable set $E$ with $\overline{\Omega_1}  \subset E \subset \overline{\Omega} \cup \overline{\Omega_1}$. By the lower semi-continuity of the total variation of BV functions with respect to $L^1$ convergence, we have that $J = Per(E,\Omega_\ell)$.  
\end{proof}

It is convenient to denote the symmetric difference of two sets $A$ and $B$ as $A \Delta B = (A \setminus B) \cup (B \setminus A)$.

\begin{lem}\label{Lem:AlMin}
Under the notations of Lemma \ref{lem:ex}, the set $E$ is a local almost minimizer of $Per_{\mathring{g}}$, namely there exist $r_1 = r_1(\Omega) > 0$ and $C = C(n,\Omega,\|\ln u\|_{L^\infty({\Omega_\ell})}) > 0$ such that
\begin{equation}
Per_{\mathring{g}}(E,B_{r}(x_0)) \leq Per_{\mathring{g}}(F,B_{r}(x_0)) + C(r + \mathop{\mathrm{ess\,osc}}_{\bar B_r(x_0) \cap \bar\Omega_\ell} \ln u) r^{n-1}
	\label{Eq:12I22-CX}
\end{equation}
for all $x_0 \in \partial E$, $r \in (0,r_1]$ and $F$ such that $\overline{F \Delta E} \subset B_{r}(x_0) \cap \Omega$.
\end{lem}

\begin{proof} Fix some $x_0 \in \partial E$ and $r_1 < \frac{1}{2}\min_{i \neq j} \textrm{dist}(\partial\Omega_i, \partial\Omega_j)$ so that the ball $B_{r_1}(x_0)$ intersects at most one component of $\partial\Omega$. We will only consider the case either $B_{r_1}(x_0) \subset \Omega$ or $B_{r_1}(x_0) \cap \partial\Omega_1 \neq \emptyset$. The other cases can be dealt with in the same way. Note that the above implies $B_{r_1}(x_0) \subset \overline{\Omega} \cup \overline{\Omega_1}$.

Let us collect a few inequalities. Since $E$ is a minimizer for the minimization problem $J$, we have
\begin{equation}
Per_g(E,B_{r}(x_0)) \leq Per_g(F \cup \overline{\Omega_1},B_{r}(x_0))
	\label{Eq:12I22-E1}
\end{equation}
for all $r \in (0,r_1)$ and $F$ such that $\overline{F \Delta E} \subset B_{r}(x_0)$. (Here we have used that $B_{r_1}(x_0) \subset \overline{\Omega} \cup \overline{\Omega_1}$.) By \eqref{Eq:12I22-E1}, there exists $C_1 > 0$ depending only on $n$ and $\|\ln u\|_{L^\infty(\bar B_{r_1}(x_0))}$ such that, for $r \in (0,r_1)$,
\begin{equation}
Per_g(E,B_{r}(x_0))  \leq \mathcal{H}^{n-1}_g(\partial B_r(x_0)) \leq C_1r^{n-1}.
\label{Eq:RegVG}
\end{equation}
By \eqref{Eq:PerInOut} and \eqref{Eq:12I22-E1}, we have
\begin{equation}
Per_g(E,B_{r}(x_0)) \leq Per_g(F,B_{r}(x_0)) + Per_g(\overline{\Omega_1},B_{r}(x_0)) - Per_g(F \cap \overline{\Omega_1},B_{r}(x_0))
	\label{Eq:12I22-E2}
\end{equation}
for all $r \in (0,r_1)$ and $F$ such that $\overline{F \Delta E} \subset B_{r}(x_0)$. 

We now prove \eqref{Eq:12I22-CX}. Take an arbitrary $F$ such that $\overline{F \Delta E} \subset B_{r}(x_0)$.

\medskip
\noindent\underline{Case 1:} If $Per_g(F,B_{r}(x_0)) > C_1r^{n-1}$, then \eqref{Eq:12I22-CX} follows from \eqref{Eq:RegVG} and Lemma \ref{Lem:PerRel}. 

\medskip
\noindent\underline{Case 2:} If $Per_g(F,B_{r}(x_0)) \leq C_1r^{n-1}$ and $Per_g(F \cap \overline{\Omega_1}, B_{r}(x_0)) > Per_g(\overline{\Omega_1},B_{r}(x_0))$, then \eqref{Eq:12I22-E2} gives
\[
Per_g(E,B_{r}(x_0)) \leq Per_g(F,B_{r}(x_0)) ,
\]
and so \eqref{Eq:12I22-CX} follows from Lemma \ref{Lem:PerRel}. 

\medskip
\noindent\underline{Case 3:} If the above two cases do not hold, then $Per_g(F,B_{r}(x_0)) \leq C_1r^{n-1}$ and $Per_g(F \cap \overline{\Omega_1}, B_{r}(x_0)) \leq Per_g(\overline{\Omega_1},B_{r}(x_0))$. By the regularity of $\partial\Omega_1$, $Per_g(\overline{\Omega_1},B_{r}(x_0)) = O(r^{n-1})$ as $r \rightarrow 0$ where the implicit constant depends only on $n$, $\Omega$ and $\|\ln u\|_{L^\infty(\bar B_{r_1}(x_0))}$. By \eqref{Eq:RegVG}, \eqref{Eq:12I22-E2} and Lemma  \ref{Lem:PerRel}, we then have
\begin{multline*}
Per_{\mathring{g}}(E,B_{r}(x_0)) \leq Per_{\mathring{g}}(F,B_{r}(x_0)) \\
	+ Per_{\mathring{g}}(\overline{\Omega_1},B_{r}(x_0)) - Per_{\mathring{g}}(F \cap \overline{\Omega_1},B_{r}(x_0))
	+ C \mathop{\textrm{osc}}_{\bar B_r(x_0) \cap \overline{\Omega_\ell}} \ln u\,r^{n-1},
\end{multline*}
where the constant $C$ depends only on $n$, $\Omega$ and $\|\ln u\|_{L^\infty(\bar B_{r_1}(x_0))}$. Since $\partial\Omega_1$ is smooth and 
\[
\overline{(F \cap \bar\Omega_1) \Delta \overline{\Omega_1}} \subset \overline{\overline{\Omega_1} \setminus F} \subset \overline{E \setminus F} \subset B_{r(x_0)},
\] 
it is known that (see \cite[Proposition 1]{Tamanini82-JRAM}) 
\[
Per_{\mathring{g}}(\overline{\Omega_1},B_{r}(x_0)) - Per_{\mathring{g}}(F \cap \overline{\Omega_1},B_{r}(x_0)) \leq Cr^n.
\]
The conclusion follows.
\end{proof}

In view of known results for almost minimizing sets, we have the following two consequences of Lemma \ref{Lem:AlMin} when the $u$ is suitably Dini continuous. In the sequel, we suppose that $\omega: (0,r_0] \rightarrow [0,\infty)$ for some $r_0 > 0$ satisfies
\begin{equation}
\frac{\omega(r)}{r} \text{ in non-increasing in $(0,r_0]$ and } \int_0^{r_0}  \big[\omega(r)\big]^{1/2} \frac{\ud r}{r} < \infty.
	\label{Eq:DiniC}
\end{equation}

\begin{cor} \label{cor:IntReg} 
Under the notations of Lemma \ref{lem:ex}, suppose further that $u$ satisfies
\[
\sup_{x_0 \in \Omega}  \mathop{\mathrm{osc}}_{B_r(x_0) \cap \Omega} \ln u \leq \omega(r) \text{ for } 0 < r < r_0 \text{ and some $\omega$ satisfying \eqref{Eq:DiniC}}.
\]
Then there exists a set $\mathcal{S}$ of Hausdorff dimension at most $n-8$ such that $\partial E \setminus \mathcal{S}$ is a $C^1$-hypersurface (containing the reduced boundary $\partial^* E$ of $E$).
\end{cor}  

\begin{proof}
This is a direct consequence of Lemma \ref{Lem:AlMin} and known regularity for almost minimizing sets (see e.g. \cite[Chapter 1]{LinThesis}, \cite[Theorem 1.9]{Tamanini84}).
\end{proof}

\begin{cor} \label{cor:BdrReg} 
Under the notations of Lemma \ref{lem:ex}, suppose further for some $\delta > 0$ that $u$ satisfies
\[
\sup_{x_0 \in \Omega,\mathrm{dist}(x,\partial\Omega) < \delta}  \mathop{\mathrm{osc}}_{B_r(x_0) \cap \Omega} \ln u \leq \omega(r) \text{ for } 0 < r < r_0 \text{ and some $\omega$ satisfying \eqref{Eq:DiniC}}.
\]
Then for every point $a \in \partial E \cap \partial\Omega$, there is a ball $B_{r}(a)$ such that $\partial E \cap B_r(a)$ is a $C^1$-hypersurface.
\end{cor}  

\begin{proof}
By \cite[Proposition 3.4]{Tamanini84}, $\partial E$ has a minimal tangent cone at $a$ which must be a plane in view of the regularity of $\partial\Omega$. This implies that $a$ belongs to the reduced boundary of $E$. By \cite[Theorem 1.9]{Tamanini84}, $\partial E$ is a $C^1$-hypersurface in a neighborhood of $a$.
\end{proof}

\begin{lem} \label{lem:BdrDetach} 
Under the notations of Lemma \ref{lem:ex}, suppose in addition that $u$ is $C^1$ near $\partial\Omega$ and that $\partial\Omega$ has non-negative mean curvature $H_{\partial\Omega} \geq 0$ with respect to $g = g_u$ and the normal pointing towards $\Omega$. Then, if the mean curvature $H_{\partial\Omega_i}$ with respect to $g$ of some component $\partial\Omega_i$ of $\partial\Omega$ satisfies $H_{\partial\Omega_i} \not\equiv 0$, the minimization problem $J$ in \eqref{Eq:ObsProb} is achieved by some set $E$ with $\pa E\cap \pa \Omega_i =\emptyset$. 
\end{lem}

\begin{proof}
By Lemma \ref{lem:ex}, the minimization problem $J$ in \eqref{Eq:ObsProb} is achieved by some set $E$. To conclude, we show that the statement $\partial E \cap \partial\Omega_i \neq \emptyset$ implies $H_{\partial\Omega_i} \equiv 0$.

Pick $a \in \partial E \cap \partial\Omega_i$. By Corollary \ref{cor:BdrReg}, $\partial E$ is a $C^1$-hypersurface near $x_0$. We then work in a local coordinate system $x = (x', x^n) = (x^1, \ldots, x^n)$ near $x_0 \cong 0$ such that the hypersurface $\partial \Omega_i$ and $\partial E$ are represented as $\{x^n = \underline{f}(x')\}$ and $\{x^n = \overline{f}(x')\}$ with $\underline{f} \in C^2$, $\overline{f} \in C^{1}$, $\underline{f}(x') \leq \overline{f}(x')$ in $\{|x'|^2 := (x^1)^2 + \ldots + (x^{n-1})^2 < r_0^2\}$, $\underline{f}(0) = \overline{f}(0) = 0$ and $\partial\underline{f}(0) = \partial\overline{f}(0) = 0$. The fact that $\partial \Omega_1$ has non-negative mean curvature and that $E$ solves the minimization problem $J$ in \eqref{Eq:ObsProb} imply
\begin{align*}
-\partial_i \partial_{p_i} e(x', \underline{f}(x'), \partial \underline{f}(x')) + \partial_s e(x', \underline{f}(x'), \partial \underline{f}(x')) &\leq 0 \text{ in }\{ |x'| < r_0\},\\
-\partial_i \partial_{p_i} e(x', \overline{f}(x'), \partial \overline{f}(x')) + \partial_s e(x', \overline{f}(x'), \partial \overline{f}(x'))  &\geq 0 \text{ in }\{ |x'| < r_0\}
\end{align*}
in the weak sense, where $e(x',\psi(x'), \partial \psi(x'))$ denotes the volume density of the metric induced by $g$ on the graph of a function $\psi$, and $s$ and $p$ are dummy variables for $\psi$ and $\partial\psi$. Since $(\partial_{p_i} \partial_{p_j} e)$ is positive definite, $\underline{f} \leq \overline{f}$ in $\{|x'| < r_0\}$ and $\underline{f}(0) = \overline{f}(0) = 0$, a standard PDE argument shows that $\underline{f} \equiv \overline{f}$ in $\{|x'| < r_0\}$, i.e. $\partial\Omega_i$ coincides with $\partial E$ and has zero mean curvature near $x_0$. A standard argument then shows that the mean curvature of $\partial\Omega_i$ is zero on all of $\partial\Omega_i$.
\end{proof}

We conclude this section with the following proposition.

\begin{prop} \label{prop:Conc} 
Let $\Omega = \Omega_\ell \setminus (\Omega_1 \cup \ldots \cup \Omega_{\ell-1})$ be a bounded domain in $\R^n$, $n \geq 3$, with smooth and disconnected boundary $\partial \Omega = \partial\Omega_1 \cup \ldots \cup \partial \Omega_\ell$, $\ell \geq 2$. Suppose $\ln u \in C^1(\bar\Omega_\ell)$ and that $\partial\Omega$ has non-negative mean curvature $H_{\partial\Omega} \geq 0$ with respect to $g = g_u$ and the normal pointing towards $\Omega$. Then, unless $\partial\Omega$ has a component $\partial\Omega_i$ with zero mean curvature $H_{\partial\Omega_i} \equiv 0$ with respect to $g$, the minimization problem $J$ in \eqref{Eq:ObsProb} is achieved by some set $E$ such that $\pa E\cap \pa \Omega =\emptyset$ and, for some set $\mathcal{S}$ of Hausdorff dimension at most $n - 8$, $\partial E \setminus \mathcal{S}$ is a $W^{2,p}$-hypersurface with zero mean curvature with respect to $g$ for every $1 \leq p < \infty$.
\end{prop}  

Note that in the above statement, the assumption $\ln u \in C^1(\bar\Omega_\ell)$ is used to defined the mean curvatures of $\partial\Omega$ and of $\partial E \setminus \mathcal{S}$.

\begin{proof}
The result follows from Corollary \ref{cor:IntReg}, Lemma \ref{lem:BdrDetach} and standard regularity theories for quasilinear elliptic equations.
\end{proof}

\subsection{An obstruction result and proof of Theorem \ref{thm:MX}}\label{SSec:3.2}

The next ingredient of the proof of Theorem \ref{thm:MX} is the following obstruction result, which first appeared in \cite{LiNg21-JMS} under stronger regularity assumptions.

\begin{lem}\label{Lem:12XI21MC}
Assume $p \geq 2$. Let $\Omega$ be an open subset of $\mathbb{R}^n$, $n \geq 3$, $f: \bar B \rightarrow \Omega$ be a $C^1 \cap W^{2,p}$ immersion of the closed $(n-1)$-dimensional unit ball $\bar B$ into $\Omega$, $\nu$ be a continuous normal field along $f(B)$ with $|\nu|_{\mathring{g}} = 1$, $\tilde g = f^* \mathring{g}$. If $u \in C^1(\Omega)$ satisfies $\lambda(-A^u) \in \bar\Gamma_2$ in $\Omega$ in the viscosity sense, then $\tilde u := u \circ f$ satisfies in the weak sense the inequality 
\begin{equation}
 \Delta_{\tilde g} \tilde u + \frac{n-2}{4(n-1)} H_{u}^2 \circ f \tilde u^{\frac{n+2}{n-2}} - \frac{n-2}{4(n-1)} H^2 \circ f \tilde u   - \frac{1}{(n-2) \tilde u}\,|\nabla_{\tilde g} \tilde u|^2
 	\geq 0 \text{ in } B,
	\label{Eq:MCIneql}
\end{equation}
where $H$ and $H_u$ denote the mean curvatures of $f(B)$ with respect to $\mathring{g}$ and $g_u$ in the direction of the normal field $\nu$, respectively. 
\end{lem}

Note that the assumption $f \in C^1 \cap W^{2,p}$ with $p \geq 2$ implies that $H \circ f \in L^2(B)$. Since $u \in C^1$, the connection form for $g_u$ is $C^0$. This together with $f \in C^1 \cap W^{2,p}$ with $p \geq 2$ imply that $H_u$ is well-defined and $H_u \circ f \in L^2(B)$. Moreover, we have
\[
\partial_\nu u + \frac{n-2}{2(n-1)} H u = \frac{n-2}{2(n-1)} H_u u^{\frac{n}{n-2}}.
\]
If $u$ is only locally Lipschitz in $\Omega$, the meaning of \eqref{Eq:MCIneql} is unclear as $H_u$ is ill-defined: the normal derivative $\partial_\nu u$, in the sense of normal trace, may not belong to any Lebesgue spaces and $\partial_\nu u$ may not be the same as $-\partial_{-\nu} u$.

\begin{cor}\label{Cor:Obs1}
Assume $p \geq 2$. Let $\Omega$ be an open subset of $\mathbb{R}^n$, $n \geq 3$, and $\Sigma$ be a $C^1 \cap W^{2,p}$ embedded hypersurface in $\Omega$ with or without boundary, $int(\Sigma)$ be its interior, $\nu$ be a continuous normal field along $int(\Sigma)$ with $|\nu|_{\mathring{g}} = 1$, $\tilde g$ be the metric induced by $\mathring{g}$ on $int(\Sigma)$. If $u \in C^1(\Omega)$ satisfies $\lambda(-A^u) \in \bar\Gamma_2$ in $\Omega$ in the viscosity sense, then $u$ satisfies in the weak sense the inequality 
\begin{equation}
 \Delta_{\tilde g} u + \frac{n-2}{4(n-1)} H_{u}^2 u^{\frac{n+2}{n-2}} - \frac{n-2}{4(n-1)} H^2 u   - \frac{1}{(n-2) u}\,|\nabla_{\tilde g} u|^2 
 	\geq 0 \text{ in } int(\Sigma),
	\label{Eq:MCCor}
\end{equation}
where $H$ and $H_u$ denote the mean curvatures of $int(\Sigma)$ with respect to $\mathring{g}$ and $g_u$ in the direction of the normal field $\nu$, respectively. 
\end{cor}

\begin{proof}[Proof of Lemma \ref{Lem:12XI21MC}] We adapt the argument in \cite{LiNg21-JMS}. 

Shrinking $\Omega$ slightly, we may assume that $u \in C^1(\tilde\Omega)$ for some open set $\tilde\Omega \supset \bar\Omega$ and $\lambda(-A^u) \in \bar\Gamma_2$ in the viscosity sense in $\tilde\Omega$. Working locally if necessary, we may also assume without loss of generality that $f$ is an embedding. We select $f_\ell \in C^2(\bar B)$ such that $f_\ell \rightarrow f$ in $C^1 \cap W^{2,p}$ and a unit normal vector field $\nu_\ell$ along $f_\ell(\bar B)$ such that $\nu_\ell \rightarrow \nu$ in $C^0 \cap W^{1,p}$.

\medskip
\noindent\underline{Claim:} There exist $u_\ell \in C^2(\bar\Omega)$ and $\varepsilon_\ell \rightarrow 0$ such that $u_\ell \rightarrow u$ in $C^1(\Omega)$ and $\lambda(-A^{u_\ell} + \varepsilon_\ell I) \in \bar\Gamma_2$ in $\Omega$. (See \cite[Section 2]{LiNgGreen} for a related discussion concerning lower Ricci curvature bounds for continuous metrics which are conformal to smooth metrics.)

\medskip
\noindent\underline{Proof of the claim:} Let $w = \frac{2}{n-2}\ln u$ so that
\[
A^u = e^{-2w}\Big[-\nabla^2 w + \nabla w \otimes \nabla w- \frac{1}{2} |\nabla w|^2 I\Big] =: e^{-2w}S(w).
\]
Let $U_2$ be the set of symmetric $n \times n$ matrices whose eigenvalues belong to $\Gamma_2$. Then $U_2$ is a convex cone with vertex at the zero matrix (see e.g. \cite[Lemma B.1]{LiNgGreen}). Let $U_2^*$ be the cone dual to $U_2$, i.e.
\[
U_2^* =\Big\{\text{ symmetric } n \times n \text{ matrices } a = (a_{ij}): \sum_{1 \leq i,j \leq n} a_{ij} b_{ij} > 0 \text{ for all } b = (b_{ij}) \in U_2\Big\}.
\]
Note that $U_2^*$ is a subcone of the cone of positive definite symmetric $n \times n$ matrices.

Arguing as in the proof of \cite[Proposition 2.5]{LiNgGreen} and using that $-\lambda(A^u) \in \bar\Gamma_2$ in the viscosity sense, we have for any $a \in C_c^\infty(\Omega; U_2^*)$ that
\begin{equation}
\int_\Omega \Big[ -\nabla_i a_{ij} \nabla_j w  - a_{ij} \nabla_i w \nabla_j w + \frac{1}{2} |\nabla w|^2 \textrm{tr}(a)\Big]\,\ud x \geq 0.
	\label{Eq:04III22-App1}
\end{equation}
(This can be viewed as the weak sense of the inclusion $-S(w) \in \bar U_2$.)

Having \eqref{Eq:04III22-App1} at hand, we can apply the proof of \cite[Proposition 2.6]{LiNgGreen} to obtain the claim. Since the situation in \cite{LiNgGreen} was more sophisticated due to the generality considered there, we provide the details here for the readers' convenience.

Let $\varrho \in C_c^\infty(\mathbb{R})$ be a non-negative even function with $|\mathbb{S}^{n-1}|\int_0^\infty \varrho(t)\,t^{n-1}\,\ud t = 1$, $\varrho_\ell(x) = \ell^n \varrho(\ell |x|)$, $w_\ell = (\varrho_\ell * w)|_{\Omega}$ (for sufficiently large $\ell$) and $u_\ell = e^{\frac{n-2}{2}w_\ell}$. 

We have, when $\textrm{dist}(x,\partial\tilde\Omega) > 1/\ell$,
\begin{align*}
\nabla_i \nabla_j w_\ell(x) = - \int_{\tilde\Omega} \nabla_{y_i}  \varrho_\ell(x - y) \nabla_{y_j} w(y)\,\ud y.
\end{align*}
Fix some $a \in C_c^\infty(\Omega; U_2^*)$. For $\ell > \textrm{dist}(\textrm{Supp}(a), \partial\Omega)^{-1}$, we have
\begin{align*}
\int_\Omega a_{ij}(x) \nabla_i \nabla_j w_\ell(x)\ud x
	&= - \int_\Omega a_{ij}(x) \int_{\tilde\Omega} \nabla_{y_i}  \varrho_\ell(x - y) \nabla_{y_j} w(y)\,\ud y\,\ud x\\
	&= - \int_{\Omega} \nabla_{y_j} w(y)  \int_{\Omega} a_{ij}(x) \nabla_{y_i}  \varrho_\ell(x - y) \,\ud x \,\ud y\\
	&= - \int_{\Omega} \nabla_{y_j} w(y)  \nabla_{y_i} a^\ell_{ij}(y)\,\ud y
\end{align*}
where $a^\ell = (a^\ell_{ij}) = (a_{ij} * \varrho_\ell)$. As $U_2^*$ is convex and $a \in C_c^\infty(\Omega; U_2^*)$, we have that $a^\ell \in C_c^\infty(\Omega; U_2^*)$. Therefore, applying \eqref{Eq:04III22-App1} to the right hand side of the above identity, we have 
\begin{align*}
&\int_\Omega a_{ij}(x) \nabla_i \nabla_j w_\ell(x)\ud x
	=  - \int_{\Omega} \nabla_{y_j} w(y)  \nabla_{y_i} a^\ell_{ij}(y)\,\ud y\\
 	&\qquad\geq \int_\Omega \Big[ a^\ell_{ij}(y) \nabla_i w(y) \nabla_j w(y) - \frac{1}{2} |\nabla w(y)|^2 \textrm{tr}(a^\ell(y))\Big]\,\ud y\\
	&\qquad= \int_\Omega \int_\Omega \varrho_\ell(y-x)\Big[ a_{ij}(x) \nabla_i w(y) \nabla_j w(y) - \frac{1}{2} |\nabla w(y)|^2 \textrm{tr}(a(x))\Big]\,\ud x\,\ud y\\
	&\qquad= \int_\Omega a_{ij}(x)  \int_\Omega \varrho_\ell(y-x)  \nabla_i w(y) \nabla_j w(y)\,\ud y\,\ud x\\
		&\qquad\qquad  - \frac{1}{2} \int_\Omega\textrm{tr}(a(x)) \int_\Omega \varrho_\ell(y-x) |\nabla w(y)|^2\,\ud y\,\ud x\\
	&\qquad= \int_\Omega \Big[a_{ij}(x)  [\varrho_\ell * (\nabla_i w \nabla _j w)](x) - \frac{1}{2} \textrm{tr}(a(x)) [\varrho_\ell * (|\nabla w|^2)](x)\Big]\,\ud x.
\end{align*}
Recalling that $w \in C^1(\tilde\Omega)$, we have that $\varrho_\ell * (\nabla_i w \nabla _j w) - \nabla_i w_\ell \nabla_j w_\ell$ and $\varrho_\ell * (|\nabla w|^2) - |\nabla w_\ell|^2$ converge uniformly in $\bar\Omega$ to zero. We hence deduce from the above inequality that 
\begin{align*}
\int_\Omega a_{ij} \nabla_i \nabla_j w_\ell\ud x
	&\geq \int_\Omega \Big[a_{ij} \nabla_i w_\ell \nabla _j w_\ell - \frac{1}{2} \textrm{tr}(a) |\nabla w_\ell|^2 - \varepsilon_\ell \textrm{tr}(a) \Big]\,\ud x
\end{align*}
for some $\varepsilon_\ell \rightarrow 0$ independent of $a$. Since $a$ is arbitrary in $C_c^\infty(\Omega; U_2^*)$ and $w_\ell$ is smooth, the above implies that $-S(w_\ell) + \varepsilon_\ell I \in \bar U_2$ in $\Omega$. Slightly adjusting $\varepsilon_\ell$ by a multiplicative factor, we have $\lambda(-A^{u_\ell} + \varepsilon_\ell I) \in \bar\Gamma_2$ in $\Omega$, which proves the claim.

To proceed, recall the following fact from \cite[Lemma 2.1]{LiNg21-JMS}: If $M$ is a symmetric $n\times n$ matrix with $\lambda(M) \in \bar\Gamma_2$ and $m \in \mathbb{R}^n$ is a unit vector, then $M_{ij}(\delta_{ij} - m_i m_j) \geq 0$. Using this fact with $M =  \frac{n-2}{2} u_\ell^{\frac{n+2}{n-2}}(-A^{u_\ell} + \varepsilon_\ell I)$ and $m = \nu_\ell$ gives
\begin{equation}
0 \leq \nabla_i \nabla_j u_\ell (\delta_{ij} - \nu_{\ell,i} \nu_{\ell,j}) - \frac{ |\nabla u_\ell|^2}{(n-2)u_\ell} + \frac{n|\partial_{\nu_\ell} u_\ell |^2}{(n-2)u_\ell }  + C\varepsilon_\ell \text{ in } f_\ell(B).
	\label{Eq:02III22-1}
\end{equation}
To proceed, we note that $\nabla_i \nabla_j u_\ell (\delta_{ij} - \nu_{\ell,i} \nu_{\ell,j})$ is the trace of the projection of $\nabla^2u$ along the tangent hyperplanes of $f_\ell(B)$, i.e. if $e_1, \ldots, e_{n-1}$ is a local orthonormal frame along $f_\ell(B)$, then
\[
\nabla_i \nabla_j u_\ell (\delta_{ij} - \nu_{\ell,i} \nu_{\ell,j}) = \sum_{j=1}^{n-1} \nabla^2 u_\ell(e_j, e_j).
\]
Letting $II_\ell$ and $H_\ell$ be the second fundamental form and the mean curvature of $f_\ell(B)$, and noting that $\nabla^2 u_\ell = \nabla^2_{f_\ell(B)} u_\ell + II_\ell \partial_{\nu_\ell} u_\ell$, we thus have
\begin{equation}
\nabla_i \nabla_j u_\ell (\delta_{ij} - \nu_{\ell,i} \nu_{\ell,j}) = \Delta_{f_\ell(B)} u_\ell + H_\ell  \partial_{\nu_\ell} u_\ell.
	\label{Eq:02III22-2}
\end{equation}
Substituting \eqref{Eq:02III22-2} into \eqref{Eq:02III22-1} we obtain
\[
0 \leq \Delta_{f_\ell(B)} u_\ell + H_\ell  \partial_{\nu_\ell} u_\ell + \frac{(n-1)|\partial_{\nu_\ell} u_\ell |^2}{(n-2) u_\ell }  - \frac{|\nabla_{f_\ell(B)} u_\ell|^2 }{(n-2) u_\ell } + C\varepsilon_\ell \text{ in } f_\ell(B).
\]
Pulling this back to $B$, we get
\begin{multline}
0 \leq \Delta_{f_\ell^* g} u_\ell \circ f_\ell + H_\ell \circ f_\ell \partial_{\nu_\ell} u_\ell \circ f_\ell\\
	 + \frac{(n-1)|\partial_{\nu_\ell} u_\ell \circ f_\ell|^2}{(n-2) u_\ell \circ f_\ell}  - \frac{|\nabla_{\tilde g} u_\ell \circ f_\ell|^2 }{(n-2) u_\ell \circ f_\ell} + C\varepsilon_\ell \text{ in } B.
	 \label{Eq:Obs-PreLimit}
\end{multline}
Sending $\ell \rightarrow \infty$ and using the relation between $H$ and $H_u$, we arrive at \eqref{Eq:MCIneql}.
\end{proof}

\begin{proof}[Proof of Theorem \ref{thm:MX}.]
Suppose by contradiction that such a function $u$ exists. By Proposition \ref{prop:Conc}, either a component $\partial \Omega_{i_0}$ of $\partial\Omega$ has zero mean curvature with respect to $g$, or there is a set $E$ such that $\partial E \subset \Omega$ and, for some set $\mathcal{S}$ Hausdorff dimension at most $n-8$, $\partial E \setminus \mathcal{S}$ is a $W^{2,p}$-hypersurface with zero mean curvature with respect to $g$ for every $1 \leq p < \infty$. (In particular, if $n \leq 7$, then $\mathcal{S}$ is empty.) We consider these cases in turn.

\medskip
\noindent\underline{Case 1:}  $\partial \Omega_{i_0}$ has zero mean curvature with respect to $g$.

Let $\Sigma = \partial\Omega_{i_0}$ and $\tilde g$ be the metric on $\Sigma$ induced by $\mathring{g}$. Since $\lda (A^u) \in \bar\Gamma_2$, by Corollary \ref{Cor:Obs1} we have 
\begin{equation} \label{eq:final0}
\Delta_{\tilde g} u-\frac{n-2}{4(n-1)} |\mathring{H}_\Sigma|^2 u -\frac{1}{(n-2) u} |\nabla_{\tilde g} u|^2 \geq 0 \quad \mbox{on }\Sigma
\end{equation}
in the weak sense where $\mathring{H}_\Sigma$ is the mean curvature of $\Sigma$ with respect to the flat ambient metric. Testing \eqref{eq:final0} against a constant function, we obtain that $u$ is constant on $\Sigma$ and $\mathring{H}_\Sigma = 0$ a.e. on $\Sigma$. This gives a contradiction since there is no smooth closed minimal hypersurfaces in Euclidean space.

\medskip
\noindent\underline{Case 2:} There is a set $E$ such that $\partial E \subset \Omega$ and, for some set $\mathcal{S}$ Hausdorff dimension at most $n-8$, $\partial E \setminus \mathcal{S}$ is a $W^{2,p}$-hypersurface with zero mean curvature with respect to $g$ for every $1 \leq p < \infty$

Let $\Sigma = \partial E$ and $\tilde g$ be the metric on $\Sigma\setminus \mathcal{S}$ induced by $\mathring{g}$. As before, we have
\begin{equation} \label{eq:final}
\Delta_{\tilde g} u-\frac{n-2}{4(n-1)} |\mathring{H}_\Sigma|^2 u -\frac{1}{(n-2) u} |\nabla_{\tilde g} u|^2 \geq 0 \quad \mbox{on }\Sigma\setminus \mathcal{S}
\end{equation}
in the weak sense.

If $\mathcal{S}$ is empty, we have as in the previous case that $\mathring{H}_\Sigma = 0$ a.e. on $\Sigma$. Since $\Sigma$ is a $W^{2,p}$-hypersurface for all $\alpha \in (0,1)$ and $p \in (1,\infty)$, the regularity theory for minimal surfaces implies that $\Sigma$ is smooth, which again gives a contradiction.

Suppose in the rest of the proof that $\mathcal{S}$ is non-empty, which implies $n\ge 8$. Since $\mathcal{S}$ has Hausdorff dimension at most $n - 8$,   for any $\epsilon>0$, one can find finite many (Euclidean) balls $ B_{r_1}(x_1), \ldots, B_{r_m}(x_m)$ with $m \geq 1$, $x_i \in \mathcal{S}$ and $0 < r_i < r_0/4$ such that $\mathcal{S} \subset \cup_{i=1}^m B_{r_i}(x_i)\subset \Omega$ 
\[
\sum_{i=1}^m r_i^{n-2} <\epsilon.
\]

Let
\[
\zeta_i(x)= \begin{cases} 
1& \quad \mbox{if } |x-x_i| \le r_i, \\
2-\frac{|x-x_i|}{r_i}& \quad  \mbox{if } |x-x_i| \le 2r_i, \\  
0& \quad  \mbox{if } |x-x_i| \ge 2r_i.
\end{cases}
\] 
Let $\zeta= \max_{1\le i \le m} \zeta_i \in C_c^{0,1}(\mathbb{R}^n)$.  Then $\zeta=0$ on $\mathbb{R}^n \setminus \cup_{i=1}^m B_{2r_i}(x_i)$, $\zeta=1$ on $\cup_{i=1}^m B_{r_i}(x_i)$ and, for every $x \in \cup_{i=1}^m B_{2r_i}(x_i) \setminus B_{r_i}(x_i)$, there is some $j = j(x)$ such that $x \in B_{2r_j}(x_j) \setminus B_{r_j}(x_j)$ and $|d\zeta_j|_{\mathring{g}} < \frac{2}{r_j}$. Thus, by \eqref{Eq:RegVG},
\begin{eqnarray*}
\Big| \int_{\Sigma\setminus \mathcal{S}} \Delta_{\tilde g} u (1-\zeta) \,\ud \textrm{vol}_{\tilde g}\Big| 
	&=& \Big| \int_{\Sigma \cap (\cup_{i=1}^m B_{2r_i}(x_i) \setminus B_{r_i}(x_i))} \tilde g(d u, d\zeta) \,\ud \textrm{vol}_{\tilde g}\Big| \\
	&\leq& C( \| u\|_{C^{0,1}(\bar\Omega)}) \sum_{i=1}^m r_i^{-1} \textrm{Vol}_{\tilde g}\big(\Sigma \cap  (B_{2r_i}(x_i) \setminus B_{r_i}(x_i)\big)  \\
	&\stackrel{\eqref{Eq:RegVG}}{\le}& C( \| u\|_{C^{0,1}(\bar\Omega)}) \sum_{i=1}^m r_i^{n-2} \le C( \| u\|_{C^{0,1}(\bar\Omega)})\epsilon.
\end{eqnarray*}

Now multiplying \eqref{eq:final} by $1-\zeta$ and integrating over $\Sigma \setminus \mathcal{S}$ and sending $\epsilon \rightarrow 0$, we get $u$ is constant on $\Sigma \setminus \mathcal{S}$ and $\mathring{H}_\Sigma = 0$ a.e. on $\Sigma \setminus \mathcal{S}$. As in the case when $\mathcal{S}$ is empty, this implies that $\Sigma \setminus \mathcal{S}$ is smooth. Take a Euclidean ball $B$ containing $\Sigma$ such that $\partial B$ touch $\Sigma$ at some point, say $x_0$. As in the proof of Corollary \ref{cor:BdrReg}, this implies that $x_0 \in \Sigma \setminus \mathcal{S}$. Since, near $x_0$, $\Sigma$ has zero mean curvature and $\partial B$ has positive mean curvature with respect to $\mathring{g}$, we obtain a contradiction to the strong maximum principle. The proof is complete.
\end{proof}

%----------------------------------------------------------------------------%

\small

%----------------------------------------------------------------------------%

\newcommand{\noopsort}[1]{}

\noindent Y.Y. Li 

\noindent Department of Mathematics, Rutgers University\\ Hill Center, Busch Campus, 110 Frelinghuysen
Road, Piscataway, NJ 08854, USA.  \\[2mm]
 \textsf{Email:yyli@math.rutgers.edu} 
 
 \bigskip

\noindent L. Nguyen

\noindent Mathematical Institute and St Edmund Hall, University of Oxford\\ Andrew Wiles
Building, Radcliffe Observatory Quarter, Woodstock Road, Oxford OX2 6GG, UK. \\[2mm]
 \textsf{Email: luc.nguyen@maths.ox.ac.uk}
 
 \bigskip

\noindent J. Xiong

\noindent School of Mathematical Sciences, Laboratory of Mathematics and Complex Systems, MOE\\
Beijing Normal University, Beijing 100875, China\\[2mm]
 \textsf{Email: jx@bnu.edu.cn}

\end{document}